\documentclass{amsart}
\usepackage{amssymb}
\usepackage{mathrsfs}
\usepackage{stmaryrd}
\usepackage{bbm}
\usepackage{color}
\usepackage{oldgerm}
\usepackage[english]{babel}
\usepackage[T1]{fontenc}
\usepackage[latin1]{inputenc}
\usepackage[all]{xy}
\usepackage{hyperref}
\usepackage[all]{xy}
\usepackage{extarrows}

\newtheorem{thm}{Theorem} [section]
\newtheorem{cor}[thm]{Corollary}
\newtheorem{lem}[thm]{Lemma}

\newtheorem{prop}[thm]{Proposition}

\theoremstyle{definition}

\theoremstyle{remark}
\newtheorem{rmk}[thm]{Remark}

\numberwithin{equation}{section}
\numberwithin{equation}{section}
\DeclareMathOperator{\ord}{ord}

\DeclareMathOperator{\Frob}{Frob}
\DeclareMathOperator{\tr}{Tr}
\DeclareMathOperator{\Sym}{Sym}
\DeclareMathOperator{\Fil}{Fil}

\usepackage{amsmath}
\usepackage{amsthm}
\usepackage{mathrsfs}
\usepackage{amsfonts}

\allowdisplaybreaks

\begin{document}
	
	\title {Symmetric power $L$-functions of a weighted hyper-Kloosterman family}
	
	\author[B.L. Wei]{Bolun Wei}
	\address{Institute for Math \& AI, Wuhan, Wuhan University, Wuhan 430072, P.R.China}
	\email{bolunwei@whu.edu.cn}

	\date{\today}

	\maketitle
	\begin{abstract}
		As a natural generalization of the classical hyper-Kloosterman sums, we study the \(k\)-th symmetric-power \(L\)-function associated with the weighted hyper-Kloosterman family
		\[
		x_1^m+x_2+\cdots+x_n+\frac{t}{x_1\cdots x_n}
		\]over a finite field of characteristic \(p\). Under the assumptions $p>2$, \(p\nmid m(mn+1)\) and \(d_k(n,m,p)=0\), we compute the cohomology of the corresponding symmetric-power Dwork complex and obtain an explicit uniform lower bound for its $q$-adic Newton polygon. The comparison polygon has slopes \(i/m\), with multiplicities determined by coefficients of a polynomial
		\[
		\frac{(1-T^m)R(T)}{1-T^{mn+1}}.
		\]The resulting symmetric power \(L\)-function is a polynomial of degree at most
		\[
		\frac{m}{mn+1}\binom{mn+k}{mn}.
		\]
	\end{abstract}

	\section{Introduction}
	The hyper-Kloosterman sums are exponential sums associated to the $n$-variable Laurent polynomial
	$$Kl_n(t,\mathbf{x})=x_1+\cdots+x_n+\frac{t}{x_1\cdots x_n},$$
	where $t$ is a parameter and $\mathbf{x}=(x_1 ,\cdots, x_n)$.
	Using the $p$-adic method, Robba \cite{Ro86} first studied its symmetric power $L$-function for the case $n=1$.  As is well known, the $k$-th symmetric power $L$-function plays a crucial role in Wan's proof of the meromorphy of the unit root $L$-function. This has motivated further study of such $L$-functions for more general families.
	
	As a generalization of the classical hyper-Kloosterman family, Haessig and Sperber \cite{HS172} introduced a  generalized family constructed by replacing $x_1+\cdots+x_n$ with $f(x)$ and $\frac{t}{x_1\cdots x_n}$ with $tx^u$.
	They showed that the symmetric power $L$-functions of this family are rational and provided estimates of their degrees and total degrees. In this paper, we give a more explicit description of $L$-functions for the following weighted hyper-kloosterman family
	\begin{equation}
		Kl_{n,m}(t;\mathbf{x})=x_{1}^{m}+x_{2}+\cdot\cdot\cdot+x_{n}+\frac{t}{x_{1}x_{2}\cdot\cdot\cdot x_{n}}
	\end{equation}
	where $t$ is a parameter and $\mathbf{x}=(x_1,\ldots,x_n)$. We describe its symmetric power $L$-function precisely.
	
	Let $p>2$ be a prime and $\mathbb{F}_q$ be the finite field with $q=p^a$ elements. Let $\overline{\mathbb{F}}_q$ be the algebraic closure of $\mathbb{F}_q$. 
	Fix $\bar{t}\in\overline{\mathbb{F}}_{q}^{*}$. Denote by $\mathbb{F}_{q_{\bar{t}}}$ the extension of $\mathbb{F}_{q}$ by adjoining $\bar{t}$.
	Then let $\deg(\bar{t})=[\mathbb{F}_{q_{\bar{t}}}:\mathbb{F}_{q}]$ and  $q_{\bar{t}}:=q^{\deg(\bar{t})}$.
	For positive integer $s$, we denote $\mathbb{F}_{q_{\bar{t}}^{s}}$ the degree $s$ extension field of $\mathbb{F}_{q_{\bar{t}}}$, and $\tr_{\mathbb{F}_{q_{\bar{t}}^{s}}/\mathbb{F}_{q}}$ the trace map of $\mathbb{F}_{q_{\bar{t}}^{s}}$ over $\mathbb{F}_{q}$.
	Let $\zeta_{p}$ be a primitive $p$-th root of unity in $\bar{\mathbb{Q}}_p$.
	Let $\Theta$ be a non-trivial additive character on $\mathbb F_q$ defined by
	$$\Theta(\bar a)=\zeta_p^{\text{Tr}_{\mathbb F_q/\mathbb F_p}(\bar a)}$$
	Define the toric exponential sum associated with $Kl_{n,m}(\bar{t},\mathbf{x})$ by
	$$S(Kl_{n,m},s,\bar{t}):=\sum_{\mathbf{x}\in(\mathbb{F}_{q_{\bar{t}}^{s}}^{*})^{n}}\Theta\circ{\rm tr}_{\mathbb{F}_{q_{\bar{t}}^{s}}/\mathbb{F}_{q}}(x_{1}^{m}+x_{2}+\cdot\cdot\cdot+x_{n}+\frac{\bar{t}}{x_{1}x_{2}\cdot\cdot\cdot x_{n}})$$
	and the associated $L$-function
	$$L(Kl_{n,m},\bar{t},T):=\exp(\sum_{s=1}^{\infty}S(Kl_{n,m},s,\bar{t})\frac{T^{s}}{s}).$$
	
	When $p\nmid m$, it follows from \cite{{AS1}} that the Laurent polynomial $Kl_{n,m}$ is nondegenerate with dimension of its Newton polytope $\triangle=n$, volume of which $Vol(\triangle)=(mn+1)/n!$. Thus by [\cite{AS1}, Corollary 3.14] we can write the $L$-function as:
	$$L(Kl_{n,m},\bar{t},T)^{(-1)^{n+1}}=(1-\pi_{0}(\bar{t})T)(1-\pi_{1}(\bar{t})T)\cdot\cdot\cdot(1-\pi_{mn}(\bar{t})T)\in\mathbb{Z}[\zeta_{p}][T].$$
	Moreover in the case where $p\equiv 1\ \mathrm{mod}\ m$, by [\cite{WY}, Theorem 1.5], we can rearrange the order of those reciprocal roots so that $\ord_{q_{\bar{t}}} \pi_{i}(\bar{t})=i/m$ for $0\leq i\leq nm$, here $\ord_{q_{\bar{t}}}$ is the normalized valuation with $\ord_{q_{\bar{t}}}(q_{\bar{t}})=1$. Using these reciprocal roots, we define the symmetric power $L$-function as follows.
	Let $k$ be a positive integer, and $p\nmid m$. The $k$-th symmetric power $L$-function of $Kl_{n,m}$ is defined by
	\begin{equation}
		L({\rm Sym}^{k}Kl_{n,m}/\mathbb{F}_{q},T):=\prod_{\bar{t}\in|\mathbb{G}_{m}/\mathbb{F}_{q}|}\prod_{\begin{subarray}{c}
				i_{0}+i_{1}+\cdot\cdot\cdot+i_{mn}=k \\
				i_{l}\geq 0\ \mathrm{for}\ 0\leq l\leq mn
		\end{subarray}}\frac{1}{1-\pi_{0}(\bar{t})^{i_{0}}\pi_{1}(\bar{t})^{i_{1}}\cdot\cdot\cdot\pi_{mn}(\bar{t})^{i_{mn}}T^{\deg(\bar{t})}},
	\end{equation}
	where the first product runs over all closed points of the algebraic torus over $\mathbb{F}_{q}$. 
	
	Using the $p$-adic method, Robba \cite{Ro86} showed the
	$k$-th symmetric power $L$-functions of $Kl_{1,1}$ are always polynomials and conjectured their degrees. Using $\ell$-adic methods,
	Fu and Wan \cite{FW1,FW2,FW3} generalized Robba's result to $Kl_{n,1}$. They proved that in most cases the $L$-function is a polynomial with integer coefficients and obtained an explicit formula of its degree. In \cite{FW1}, they asked whether a uniform lower bound exists for the Newton polygon of these $L$-functions. Haessig \cite{H17} answered this question by using Dwork theory for the case $Kl_{1,1}$. Fres\'{a}n, Sabbah and Yu \cite{FSY} subsequently gave a Hodge-theoretic interpretation in the classical setting by constructing the relevant motives and computing their Hodge numbers via irregular Hodge theory.
	For the classical hyper-Kloosterman family $Kl_{n,1}$, Haessig and Sperber \cite{HS24} proved that this $L$-function also has a uniform lower bound.
	
	For the $k$-th symmetric power $L$-functions of this weighted hyper-kloosterman family $Kl_{n,m}(t,\mathbf{x})$, we will determine their degrees  and show their $q$-adic Newton polygons have uniform lower bounds under certain restrictions, which can be viewed as a continuation of `Newton above Hodge' principle.
	To state our results explicitly, we first introduce some notations.
	
	Suppose $p\nmid mn+1$, let $\bar{\zeta}_{mn+1}$ be a primitive $(mn+1)$-th root of unity in $\overline{\mathbb{F}}_{p}$ and define
	\begin{equation} \label{(1.3)}
		\begin{array}{c}
			I_{k}:=\{\underline{i}=(i_{0},\cdot\cdot\cdot,i_{mn})\in\mathbb{Z}_{\geq0}^{mn+1}|i_{0}+i_{1}+\cdot\cdot\cdot+i_{mn}=k\}  \\ \\
			d_{k}(n,m,p):=
			\#\{\underline{i}=(i_{0},\cdot\cdot\cdot,i_{mn})\in I_{k}|\sum_{j=0}^{mn}i_{j}\bar{\zeta}_{mn+1}^{j}= 0\ \mathrm{in}\ \overline{\mathbb{F}}_{p}\}.
		\end{array}
	\end{equation}
	Next, set
	$$R(T):=\sum_{\underline{i}\in I_{k}}T^{\sum_{j=0}^{mn}j\cdot i_{j}}$$
	write
	$$\frac{R(T)(1-T^{m})}{1-T^{mn+1}}=\sum_{i\geq 0}h_{i}(n,m,k)T^{i}.$$
	Under the assumption $d_{k}(n,m,p)=0$, one note that $h_{i}(n,m,k)$ are actually nonnegative integers and infinitely many of which equals zero by the remark \ref{rmk4.5}. Let $q^{1/m}$ be a fixed $m$-th root of $q$. We then write down our main result.
	
	\begin{thm} \label{thm011}
		Let $p>2$ be a prime number and $p\nmid m(mn+1)$, $d_{k}(n,m,p)=0$. Then $L({\rm Sym}^{k}Kl_{n,m}/\mathbb{F}_{q},T)$ is a polynomial in $1+T\mathbb{Z}[\zeta_{p}][T]$ of degree at most $\frac{m}{mn+1}\binom{mn+k}{mn}$, and its $q$-adic Newton polygon lies on or above the $q$-adic Newton polygon of $\prod(1-q^{i/m}T)^{h_{i}(n,m,k)}$.
	\end{thm}
	
	Define the combinatoric Hodge polygon as the lower convex hull of the points $(0,0)$ and 
	\begin{equation} \label{Hodge}
		(\sum_{i=0}^N h_i(n,m,k),\sum_{i=0}^N \frac{i}{m}h_i(n,m,k))\ {\rm for}\ N=0,1,2,....   
	\end{equation}
	It can be seen that this coincides with the $q$-adic Newton polygon of the product 
	$$\prod(1-q^{i/m}T)^{h_{i}(n,m,k)}.$$
	Thus, our result can be interpreted as an instance of `Newton over Hodge' principle.
	
	This paper is organized as follows. In Section 2, we will construct the $p$-adic relative cohomology associated to the weighted hyper-Kloosterman sum.
	In Section 3, we construct a $p$-adic relative cohomology to express the $k$-th symmetric power $L$-function.
	In Section 4, we give the proof of Theorem \ref{thm011}.

	\section{Relative cohomology\label{Section 2}}
	\subsection{Relative cohomology\label{Subsection 2.1}}
	
	In this section we review the construction of the relative cohomology that has been used to express the classical $L$-function $L(Kl_{n,m},\bar{t},T)$. Throughout this paper, unless otherwise specified, we always have $p>2$ and $p\nmid m(mn+1)$.
	
	For $q=p^{a}$, let $\mathbb{Q}_{q}$ be the unramified extension of $\mathbb{Q}_{p}$ of degree $a$ and $\mathbb{Z}_{q}$ denote its ring of integers.
	Let
	$$E(T):=\exp(\sum_{j\geq 0}T^{p^{j}}/p^{j})\in(\mathbb{Q}\cap\mathbb{Z}_{p})[[T]]$$
	be the Artin-Hasse series and let $\gamma\in\overline{\mathbb{Q}}_{p}$ be a root of $\sum_{j\geq 0}T^{p^{j}}/p^{j}=0$ such that $\ord_{p}\gamma=1/(p-1)$. We fix $\gamma^{1/mq}$, a $mq$-th root of $\gamma$. Let $\mathfrak{D}_{q}$ be the ring of integers of $\mathbb{Q}_{q}(\gamma^{1/mq})$. 
	Recall the Newton polytope of $Kl_{n,m}$, $\Delta\in\mathbb{R}^n$, is the convex hull spanned by the points $(m,0,\cdot\cdot\cdot,0)$, $(0,1,0,\cdot\cdot\cdot,0)$,..., $(0,0,\cdot\cdot\cdot,0,1)$ and $(-1,-1,\cdot\cdot\cdot,-1)$.

	For $u=(u_{1},\cdot\cdot\cdot,u_{n})\in\mathbb{Z}^{n}$, define
	$$
	M(u)=\max\{0,-u_{1},\cdot\cdot\cdot,-u_{n}\} $$
	and
	\begin{equation} \label{weight}
		\omega(u)=\frac{u_{1}}{m}+u_{2}+\cdot\cdot\cdot+u_{n}+\frac{mn+1}{m}M(u).
	\end{equation}
	One can see that $\omega$ is actually the weight function associated with $\Delta$. In another words, for any $u\in\mathbb{Z}^{n}$, $\omega(u)$ is the smallest nonnegative real number such that $u$ lies inside the dilated polytope $\omega(u)\triangle$ including its boundary. Then the following properties hold:
	\begin{prop} \label{weight}
		For any $u,v\in\mathbb{Z}^{n}$, and $r\in\mathbb{Z}_{\geq 0}$, we have\\
		(1) $\omega(u)\ge 0$, and the equality holds if and only if $u=0$.\\
		(2) $\omega(ru)=r\omega(u)$.\\
		(3) $\omega(u+v)\leq \omega(u)+\omega(v)$ with equality if and only if $u,v$ are cofacial with respect to $\Delta$. \\
		(4) (2) and (3) hold for $M(u)$.
	\end{prop}
	\par Let $\theta(T):=E(\gamma T)=\sum_{j\geq 0}\theta_{j}T^{j}$ be Dwork's splitting function satisfying $\ord_{p}\theta_{j}\geq j/(p-1)$ and $\theta(1)=\zeta_{p}$ for the fixed primitive $p$-th root of unity. We then define
	\begin{equation}
		F(t,x):=\theta(x_{1}^{m})\theta(x_{2})\cdot\cdot\cdot\theta(x_{n})\theta(\frac{t}{x_{1}\cdot\cdot\cdot x_{n}}).
	\end{equation}
	We write $F(t,x)=\sum_{u\in\mathbb{Z}^{n}}B(u,t)x^{u}$ with
	$$B(u,t)=\sum_{(k_{i},l)\in I(u)}\theta_{k_{1}}\theta_{k_{2}}\cdot\cdot\cdot\theta_{k_{n}}\theta_{l}t^{l}$$
	where
	$$I(u)=\{(k_{i},l)\in\mathbb{Z}_{\geq 0}^{n+1}|mk_{1}-l=u_{1},\ k_{i}-l=u_{i}\ \mathrm{for}\ 2\leq i\leq n\}.$$
	If we write $B(u,t)=\sum_{l\geq M(u)}B(u;l)t^{l}$ with
	\begin{equation}
		B(u;l):=\sum_{(k_{i},l)\in I(u)}\theta_{k_{1}}\theta_{k_{2}}\cdot\cdot\cdot\theta_{k_{n}}\theta_{l},
	\end{equation}
	then we have the following estimates.
	
	\begin{prop} \label{Proposition 2.2}
		For $u\in\mathbb{Z}^{n}$ and $l\geq M(u)$, we have $${\rm ord}_{p}B(u;l)\geq\frac{\omega(u)}{p-1}+\frac{(mn+1)(l-M(u))}{m(p-1)}.$$
	\end{prop}
	\textit{Proof.} This result follows from Proposition \ref{weight} and the fact $\ord_{p}\theta_{j}\geq j/(p-1).$  $\hfill\square$

	Let $s$ be an integer and $0\leq s\leq a$. Define the spaces
	$$\mathcal{O}_{p^{s}}=\{\zeta:=\sum_{r\geq 0}\zeta(r)\gamma^{\frac{(mn+1)r}{mp^{s}}}t^{r}|\zeta(r)\in\mathfrak{D}_q,|\zeta(r)|_{p}\rightarrow 0\ \mathrm{as}\ r\rightarrow\infty\}$$
	and $$\mathcal{C}_{p^{s}}=\{\xi:=\sum_{u\in\mathbb{Z}^{n}}\xi(u)\gamma^{\omega(u)}t^{p^{s}M(u)}x^{u}|\xi(u)\in\mathcal{O}_{p^{s}},\|\xi(u)\|\rightarrow 0\ \mathrm{as}\ \omega(u)\rightarrow\infty\}.$$
	The norm on $\mathcal{O}_{p^{s}}$ is defined as $\|\zeta\|:=\sup_{r\geq 0}|\zeta(r)|_{p}$, and  the norm on $\mathcal{C}_{p^{s}}$ is $\|\xi\|:=\sup_{u\in\mathbb{Z}^{n}}\|\xi(u)\|$. The term $\gamma^{1/mp^{s}}$ is well-defined viewing as $(\gamma^{1/mq})^{p^{a-s}}$ for the fixed uniformizer $\gamma^{1/mq}$. We abbreviate $\mathcal{O}$ and $\mathcal{C}$ for the case $s=0$, respectively.
	
	For $i\geq 0$, let $\gamma_{i}:=\sum_{j=0}^{i}\gamma^{p^{j}}/p^{j}=-\sum_{j\geq i+1}\gamma^{p^{j}}/p^{j}$. The latter shows that
	\begin{equation} \label{(2.4)}
		\ord_{p}\gamma_{i}=\frac{p^{i+1}}{p-1}-i-1.
	\end{equation}
	Let
	\begin{equation} \label{H}
		H(t,x):=\sum_{j=0}^{\infty}\gamma_{j}Kl_{n,m}(t^{p^{j}},x^{p^{j}})
	\end{equation}
	and observe that $\exp(H(t,x))=\prod_{j=0}^{\infty}F(t^{p^{j}},x^{p^{j}}).$ For each $1\leq l\leq n$, we define the differential operators
	$$D_{t^{p^{s}},l}:=\exp(-H(t^{p^{s}},x))\circ x_{l}\frac{\partial}{\partial x_{l}}\circ\exp(H(t^{p^{s}},x))=x_{l}\frac{\partial}{\partial x_{l}}+x_{l}\frac{\partial H(t^{p^{s}},x)}{\partial x_{l}}.$$
	We have the following explicit expression:
	\begin{equation}  \label{differential}
		D_{t^{p^{s}},l}=\left\{\begin{array}{cc}
			x_{1}\frac{\partial}{\partial x_{1}}+\sum\limits_{j=0}^{\infty}\gamma_{j}p^{j}(mx_{1}^{mp^{j}}-\frac{t^{p^{s+j}}}{x_{1}^{p^{j}}\cdot\cdot\cdot x_{n}^{p^{j}}})   &  \mathrm{for}\ l=1\\
			x_{l}\frac{\partial}{\partial x_{l}}+\sum\limits_{j=0}^{\infty}\gamma_{j}p^{j}(x_{l}^{p^{j}}-\frac{t^{p^{s+j}}}{x_{1}^{p^{j}}\cdot\cdot\cdot x_{n}^{p^{j}}})  &  \mathrm{for}\ 2\leq l\leq n.
		\end{array}\right.
	\end{equation}
	Note that $D_{t^{p^{s}},l}$ are endomorphisms of $\mathcal{C}_{p^{s}}$ over $\mathcal{O}_{p^{s}}$ and commute with each other. Then we construct a Koszul complex $\Omega^{\bullet}(\mathcal{C}_{p^{s}},D_{t^{p^{s}}})$ by letting
	\begin{equation*}
		\Omega^{i}(\mathcal{C}_{p^{s}},D_{t^{p^{s}}}):=\left\{\begin{array}{cc}
			\mathcal{C}_{p^{s}}  &  \mathrm{for}\ i=0\\
			\bigoplus\limits_{1\leq j_{1}<j_{2}\cdot\cdot\cdot<j_{i}\leq n}\mathcal{C}_{p^{s}}\frac{dx_{j_{1}}}{x_{j_{1}}}\wedge\cdot\cdot\cdot\wedge\frac{dx_{j_{i}}}{x_{j_{i}}}  &  \mathrm{for}\ 1\leq i\leq n
		\end{array}\right.
	\end{equation*}
	with the boundary map defined by $D_{t^{p^{s}}}:\Omega^{i}\rightarrow\Omega^{i+1}$
	\begin{equation*}
		D_{t^{p^{s}}}(\xi\frac{dx_{j_{1}}}{x_{j_{1}}}\wedge\cdot\cdot\cdot\wedge\frac{dx_{j_{i}}}{x_{j_{i}}}):=\Big(\sum_{l=1}^{n}D_{t^{p^{s}},l}(\xi)\frac{dx_{l}}{x_{l}}\Big)\wedge\frac{dx_{j_{1}}}{x_{j_{1}}}\wedge\cdot\cdot\cdot\wedge\frac{dx_{j_{i}}}{x_{j_{i}}} .
	\end{equation*}
	Denote by $H^{i}(\mathcal{C}_{p^{s}},D_{t^{p^{s}}})$ the $i$-th cohomology space of $\Omega^{\bullet}(\mathcal{C}_{p^{s}},D_{t^{p^{s}}})$. 
	It follows from \cite[Theorem 3.1]{HS172} that when $p\nmid m$, the cohomology is acyclic except in the top dimension and $H^{n}(\mathcal{C}_{p^{s}},D_{t^{p^{s}}})$ is a free $\mathcal{O}_{p^s}$-module of rank $mn+1$.
	It follows from \cite[Theorem 1.2]{WY} that
	$$\{1,\gamma^{i_1/m}x_1^{i_1},\gamma^{(i_1+m)/m}x_1^{i_1}x_2,\gamma^{(i_{1}+2m)/m}x_{1}^{i_{1}}x_{2}x_{3},\cdots \gamma^{(i_1+nm-m)/m}x_1^{i_1}x_2\cdots, x_n\}_{1\le i_1\le m}$$ 
	forms a basis of $H^{n}(\mathcal{C}_{p^{s}},D_{t^{p^{s}}})$.
	We order the basis by
	\begin{equation} \label{basis}
		e_i =\begin{cases}
			1   & \mathrm{when}\ i=0, \\
			\gamma^{\frac{i}{m}} x_1^{i_1}x_2\cdots x_{\alpha+1} &\mathrm{when}\ i=i_1 +\alpha m\ \mathrm{for}\ 1\leq i_1 \leq m,\ 0\leq \alpha\leq n-1 ,
		\end{cases}    
	\end{equation}
	and denote $\bar{e}_i=\gamma^{-i/m} e_i$. For $u\in\mathbb{Z}^{n}$, abuse the notation, we sometimes write $\omega(x^u)$ for the weight of its exponent vector. By the definition of the weight function (\ref{weight}), we readily compute $\omega(\bar{e}_{i})=i/m$ for $0\leq i\leq mn$.

	
	Define $\psi_{\mathbf x}$ acting on power series ring involving $t$ and $\mathbf x$ by
	$$\psi_{\mathbf x}\big(\sum_{\mathbf{v}\in\mathbb{Z}^{n},\ r\geq 0} a(r, \mathbf v)t^r\mathbf x^{\mathbf v}\big)=\sum_{\mathbf{v}\in\mathbb{Z}^{n},\ r\geq 0} a(r, p\mathbf v)t^r\mathbf x^{\mathbf v}.$$
	Define the relative Frobenius map $\alpha_1(t^{p^s}):\mathcal{C}_{p^s}\rightarrow \mathcal{C}_{p^{s+1}}$ as
	$$
	\alpha_{1}(t^{p^{s}}):=\exp(-H(t^{p^{s+1}},x))\circ\psi_{x}\circ\exp(H(t^{p^{s}},x))=\psi_{x}\circ F(t^{p^{s}},x) $$
	and define $\alpha_a(t):\mathcal{C}\rightarrow \mathcal{C}_{q}$ as
	$$\alpha_{a}(t):=\alpha_{1}(t^{p^{a-1}})\cdot\cdot\cdot\alpha_{1}(t^{p})\alpha_{1}(t)=\exp(-H(t^{q},x))\circ\psi_{x}^{a}\circ\exp(H(t,x)).
	$$
	We also have $pD_{t^{p^{s+1}},l}\circ\alpha_{1}(t^{p^s})=\alpha_{1}(t^{p^s})\circ D_{t^{p^s},l}$ for $1\leq l\leq n$. 
	Then we can define a chain map $\Frob^{\bullet}(\alpha_{1}(t^{p^s})):\Omega^{\bullet}(\mathcal{C}_{p^s},D_{t^{p^s}})\rightarrow\Omega^{\bullet}(\mathcal{C}_{p^{s+1}},D_{t^{p^{s+1}}})$ by
	$$\Frob^{i}(\alpha_{1}(t^{p^s}))(\zeta\frac{dx_{j_{1}}}{x_{j_{1}}}\wedge\cdot\cdot\cdot\wedge\frac{dx_{j_{i}}}{x_{j_{i}}}):=p^{n-i}\alpha_{1}(t^{p^s})(\zeta)\frac{dx_{j_{1}}}{x_{j_{1}}}\wedge\cdot\cdot\cdot\wedge\frac{dx_{j_{i}}}{x_{j_{i}}}.$$
	Furthermore, we have
	$qD_{t^{q},l}\circ\alpha_{a}=\alpha_{a}\circ D_{t,l}$ for $1\leq l\leq n$.
	Then we can similarly define a chain map $\Frob^{\bullet}(\alpha_{a}):\Omega^{\bullet}(\mathcal{C},D_{t})\rightarrow\Omega^{\bullet}(\mathcal{C}_{q},D_{t^{q}})$.
	Note that $H^i=0$ for each $i\neq n$. We only consider the induced Frobenius  $\bar{\alpha}_{a}(t):=\Frob^{n}(\alpha_{a}):H^{n}(\mathcal{C},D_{t})\rightarrow H^{n}(\mathcal{C}_{q},D_{t^{q}})$.
	
	For $\bar{t}\in\overline{\mathbb{F}}_{q}^{*}$ of degree $d$, let $\hat{t}$ be its Teichm\"uller lifting of $\bar{t}$ in $\overline{\mathbb Q}_p$.
	We regard the specialization $\mathcal{C}(\hat{t})$ as an $\mathcal{O}$-algebra via the homomorphism
	$$\mathcal{O}\to \mathcal{C}(\hat{t}),\ \ \ \ t\to \hat{t}.$$
	The specialization of the induced Frobenius is $\bar{\alpha}_{a}(\hat{t})$, note that $$\bar\alpha_{ad}(\hat{t}):=
	\bar\alpha_a(\hat t^{\,q^{d-1}})
	\circ\cdots\circ
	\bar\alpha_a(\hat t^{\,q})
	\circ
	\bar\alpha_a(\hat t)$$
	acts on $H^n\otimes_{\mathcal{O}}\mathcal{C}(\hat{t})$ linearly since $\hat{t}^{q^{d}}=\hat{t}$, here $H^n$ is the abbreviation of the specialization of $H^n (\mathcal{C}_q ,D_{t^{q}})$ via $t\to \hat{t}$. By the Dwork trace formula \cite[Theorem 2.2]{AS1}, we have
	$$L(Kl_{n,m},\bar{t},T)^{(-1)^{n+1}}=\det(1-\bar{\alpha}_{a d}T|H^{n} \otimes_{\mathcal{O}}\mathcal{C}(\hat{t})).$$
	Given $p\nmid m$, it follows from \cite{{AS1}} that
	$$L(Kl_{n,m},\bar{t},T)^{(-1)^{n+1}}=(1-\pi_{0}(\bar{t})T)(1-\pi_{1}(\bar{t})T)\cdot\cdot\cdot(1-\pi_{mn}(\bar{t})T)\in\mathbb{Z}[\zeta_{p}][T].$$
	
	\subsection{$k$-th symmetric power cohomology}\label{Section 3.1}

	Define
	$$\nabla_{t^{p^{s}}}:=\exp(-H(t^{p^{s}},x))\circ t\frac{\partial}{\partial t}\circ\exp(H(t^{p^{s}},x))=t\frac{\partial}{\partial t}+t\frac{\partial H(t^{p^s},x)}{\partial t}.$$
	Observe that $\nabla_{t^{p^{s}}}$ commutes with $D_{t^{p^{s}},l}$ for $1\leq l\leq n$. So it induces a connection map on $H^{n}(\mathcal{C}_{p^{s}},D_{t^{p^{s}}})$, for which we still denote by $\nabla_{t^{p^{s}}}$. For the case $s=0$ we denote $\nabla_t$ by $\nabla$.

	Define $\mathcal{S}_{k,t^{p^{s}}}:={\rm Sym}^{k}_{\mathcal{O}_{p^{s}}}H^{n}(\mathcal{C}_{p^{s}},D_{t^{p^{s}}})$ to be the $k$-th symmetric power of the free module $H^{n}(\mathcal{C}_{p^{s}},D_{t^{p^{s}}})$. It is a free $\mathcal{O}_{p^{s}}$-module with basis
	$${\rm Sym}^{k}(\mathcal{B}):=\{e^{\underline{i}}=e_{0}^{i_{0}}e_{1}^{i_{1}}\cdot\cdot\cdot e_{mn}^{i_{mn}}:|\underline{i}|=k\},$$
	where $\underline{i}=(i_{0},\cdot\cdot\cdot,i_{mn})\in \mathbb{Z}_{\ge 0}^{mn+1}$ and $|\underline{i}|=i_0+\cdots+i_{mn}$.
	Every element of \(S_{k,t^{p^s}}\) is a finite \(\mathcal O_{p^s}\)-linear combination of the monomials \(e^{\underline{i}}\). We extend the connection by the Leibniz rule:
	$$\nabla_{t^{p^{s}}}(\zeta_{\underline{i}} e_{0}^{i_0} \cdots e_{mn}^{i_{mn}}):=t\frac{d \zeta_{\underline{i}}}{dt}e^{\underline{i}}+\sum_{j=0}^{mn} i_j\zeta_{\underline{i}} e_0^{i_0}\cdots e_j^{i_j-1}\cdots e_{nm}^{i_{nm}}\cdot \nabla_{t^{p^{s}}}(e_j).$$
	We define a complex $\Omega^{\bullet}(\mathcal{S}_{k,t^{p^{s}}},\nabla_{t^{p^{s}}})$ by
	$$\Omega^{0}:=\mathcal{S}_{k,t^{p^{s}}} \ {\rm and}\  \Omega^{1}:=\mathcal{S}_{k,t^{p^{s}}}\frac{dt}{t}$$
	with the boundary map $\nabla_{t^{p^{s}}}\xi=\nabla_{t^{p^{s}}}(\xi)\frac{dt}{t}$.
	Denote its cohomology by $H^{i}(\mathcal{S}_{k,t^{p^{s}}},\nabla_{t^{p^{s}}})$ for $i=0,1$.
	
	Define ${\rm Sym}^{k}\bar{\alpha}_{1}(t^{p^{s}}):\mathcal{S}_{k,t^{p^s}}\rightarrow \mathcal{S}_{k,t^{p^{s+1}}}$ by
	$${\rm Sym}^{k}\bar{\alpha}_{1}(t^{p^{s}})(\zeta_{\underline{i}} e_0^{i_0}e_1^{i_1}\cdots e_{mn}^{i_{mn}})=\zeta_{\underline{i}} \bar{\alpha}_{1}(t^{p^{s}})(e_0)^{i_0}\cdots \bar{\alpha}_{1}(t^{p^{s}})(e_{mn})^{i_{mn}}.$$
	Similarly we can define ${\rm Sym}^{k}(\bar{\alpha}_{a}):\mathcal{S}_{k,t}\rightarrow\mathcal{S}_{k,t^{q}}$. For $1\le s \le a$, we define $\psi_{t}:\mathcal{O}_{p^{s}}\rightarrow\mathcal{O}_{p^{s-1}}$ by $\psi_{t}: \sum a(r)t^{r}\mapsto\sum a(pr)t^{r}$. Thus we can define $\mathfrak{D}_q$-linear endomorphisms of $\mathcal{S}_{k,t}$ by
	$$\beta_{k,1}:=\psi_{t}\circ{\rm Sym}^{k}(\bar{\alpha}_{1}(t))\ {\rm and} \ \beta_{k,a}:=\psi_{t}^{a}\circ{\rm Sym}^{k}(\bar{\alpha}_{a}).$$

	\begin{lem} \label{Lemma 3.1}
		We have that $\beta_{k,a}=\beta_{k,1}^{a}$.
	\end{lem}
	\begin{proof}
		It immediately follows from the observation that
		\begin{align*}
			\beta_{k,1}^{a}&=(\psi_{t}\circ{\rm Sym}^{k}(\bar{\alpha}_{1}(t)))\circ\cdots\circ(\psi_{t}\circ{\rm Sym}^{k}(\bar{\alpha}_{1}(t))) \\ &=\psi_{t}^{a}\circ{\rm Sym}^{k}(\bar{\alpha}_{1}(t^{p^{a-1}})\cdots\bar{\alpha}_{1}(t^{p})\circ\bar{\alpha}_{1}(t))\\
			&=\beta_{k,a}.
		\end{align*}
	\end{proof}
	
	Observe that $\beta_{k,a}\circ\nabla=q\nabla\circ\beta_{k,a}$. Therefore $\beta_{k,a}$ induces an endomorphism $\bar{\beta}_{k,a}$ on $H^{i}(\mathcal{S}_{k,t},\nabla)$ for $i=0,1$. By the same argument as \cite[Theorem 3.1]{HS24}, or see \cite{H14,H17} for further details, we conclude that $\beta_{k,a}$ and $\bar{\beta}_{k,a}$ are completely continuous operators together with the following cohomological expression for the $k$-th symmetric power $L$-function
	\begin{equation} \label{(3.2)}
		L({\rm Sym}^{k}Kl_{n,m}/\mathbb{F}_{q},T)=\frac{\det(1-\bar{\beta}_{k,a}T|H^{1}(\mathcal{S}_{k,t},\nabla))}{\det(1-q\bar{\beta}_{k,a}T|H^{0}(\mathcal{S}_{k,t},\nabla))}  .
	\end{equation}
	In the next section, we will give the explicit description of the cohomology space $H^{0}(\mathcal{S}_{k,t},\nabla)$ and $H^{1}(\mathcal{S}_{k,t},\nabla)$.
	
	\section{Symmetric power cohomology}\label{Section 3}
	
	\subsection{Relative cohomology} \label{Subsetion 2.2}
	Let $\overline{\mathcal{O}}_{p^{s}}:=\mathbb{F}_{q}[t]$ and $\overline{\mathcal{C}}_{p^{s}}$ the reduced semigroup algebra
	\begin{equation*}
		\overline{\mathcal{C}}_{p^{s}}:=\bigoplus_{u\in\mathbb{Z}^{n}}\overline{\mathcal{O}}_{p^s}t^{p^s M(u)}x^u \subset \mathbb{F}_{q}[t,x_1^{\pm},\cdots, x_n^{\pm}]
	\end{equation*}
	with the usual multiplication:
	\begin{equation*}
		\bigl(t^r t^{p^sM(u)}x^u\bigr)
		\bigl(t^{r'}t^{p^sM(v)}x^v\bigr)
		=t^{r+r'+p^{s}\delta'(u,v)}t^{p^sM(u+v)}x^{u+v}
	\end{equation*}
	where $\delta'(u,v)=M(u)+M(v)-M(u+v)$. We define the reduction maps
	$\mathrm{Pr}:\mathcal{O}_{p^{s}}\rightarrow\overline{\mathcal{O}}_{p^{s}}$ by
	$$\mathrm{Pr}(\zeta(r)\gamma^{\frac{(mn+1)r}{mp^{s}}}t^{r})=\bar{\zeta}(r)t^{r},$$
	where $\bar{\zeta}(r)\in\mathbb{F}_{q}$, and
	$\mathrm{Pr}:\mathcal{C}_{p^{s}}\rightarrow\overline{\mathcal{C}}_{p^{s}}$
	by
	$${\rm Pr}(\xi(u)\gamma^{\omega(u)}t^{p^{s}M(u)}x^{u})= \bar{\xi}(u)t^{p^{s}M(u)}x^{u}$$
	with $\xi(u)\in \mathcal{O}_{p^s}$ and $\bar{\xi}(u)\in \overline{\mathcal{O}}_{p^s}$.
	Above reduction maps are both ring homomorphisms under their corresponding multiplications. The latter one gives an isomorphism $\mathcal{C}_{p^{s}}/\gamma^{1/mq}\mathcal{C}_{p^{s}}\simeq\overline{\mathcal{C}}_{p^{s}}$.
	Hence we can view the above two maps as $\mathrm{mod}\ \gamma^{1/mq}$.
	For $t^{r+p^sM(u)}x^u\in \mathcal{C}_{p^s}$, we define its total weight function by
	\begin{equation}
		W_{p^{s}}(r;u):=\frac{(mn+1)r}{mp^{s}}+\omega(u).   \label{total}
	\end{equation}
	
	On $\overline{\mathcal{O}}_{p^{s}}$, we define an increasing weighted filtration
	$$\mathrm{Fil}^{d}\overline{\mathcal{O}}_{p^{s}}:=\{\mathbb{F}_{q}\text{-}\mathrm{vector}\ \mathrm{space}\ \mathrm{spanned}\ \mathrm{by}\ t^{r}\ {\rm such\ that}\ \frac{(mn+1)r}{mp^{s}}\leq d\},$$
	and on $\overline{\mathcal{C}}_{p^{s}}$ we define
	$$\mathrm{Fil}^{N}\overline{\mathcal{C}}_{p^{s}}:=\{\mathbb{F}_{q}\text{-}\mathrm{vector}\ \mathrm{space}\ \mathrm{spanned}\ \mathrm{by}\ t^{r+p^{s}M(u)}x^{u}\ \mathrm{such\ that}\ W_{p^{s}}(r;u)\leq N\}.$$
	We also define the weighted filtrations on $\mathcal{O}_{p^{s}}$ and
	$\mathcal{C}_{p^{s}}$ as follows:
	$$\mathrm{Fil}^{d}\mathcal{O}_{p^{s}}:=\{\mathfrak{D}_q\text{-}\mathrm{module}\  \mathrm{generated}\ \mathrm{by}\ \gamma^{\frac{(mn+1)r}{mp^{s}}}t^{r}\ \mathrm{such\ that}\ \frac{(mn+1)r}{mp^{s}}\leq d\}$$
	and
	$$\mathrm{Fil}^{N}\mathcal{C}_{p^{s}}:=\{\mathfrak{D}_q\text{-}\mathrm{module}\  \mathrm{generated}\ \mathrm{by}\ \gamma^{\frac{(mn+1)r}{mp^{s}}+\omega(u)}t^{r+p^{s}M(u)}x^{u}\ \mathrm{such\ that}\ W_{p^{s}}(r;u)\leq N\}.$$
	Since triangle inequality in Proposition \ref{weight} also holds for the total weight function $W_{p^{s}}(r;u)$, we can easily obtain the following lemma:
	\begin{lem} \label{Lemma 2.5}
		For any $N_{1},N_{2}\in\frac{1}{mp^{s}}\mathrm{Z}_{\geq 0}$, we have $\mathrm{Fil}^{N_{1}}\mathcal{C}_{p^{s}}\cdot\mathrm{Fil}^{N_{2}}\mathcal{C}_{p^{s}}\subset\mathrm{Fil}^{N_{1}+N_{2}}\mathcal{C}_{p^{s}}$ and $Fil^{N_1}\overline{\mathcal{C}}_{p^{s}}\cdot Fil^{N_2}\overline{\mathcal{C}}_{p^{s}}\subseteq Fil^{N_1+N_2}\overline{\mathcal{C}}_{p^{s}}$.
	\end{lem}

	Let
	\begin{equation}
		\overline{D}_{t^{p^{s}},l}^{(1)}:=\left\{\begin{array}{cc}
			x_{1}\frac{\partial}{\partial x_{1}}+(mx_{1}^{m}-\frac{t^{p^{s}}}{x_{1}\cdot\cdot\cdot x_{n}})  & \mathrm{for}\ l=1 \\
			x_{l}\frac{\partial}{\partial x_{l}}+(x_{l}-\frac{t^{p^{s}}}{x_{1}\cdot\cdot\cdot x_{n}})  & \mathrm{for}\ 2\leq l\leq n
		\end{array}\right.
	\end{equation}
	and
	\begin{equation} \label{differential (1)}
		D_{t^{p^{s}},l}^{(1)}:=\left\{\begin{array}{cc}
			x_{1}\frac{\partial}{\partial x_{1}}+\gamma(mx_{1}^{m}-\frac{t^{p^{s}}}{x_{1}\cdot\cdot\cdot x_{n}})  & \mathrm{for}\ l=1 \\
			x_{l}\frac{\partial}{\partial x_{l}}+\gamma(x_{l}-\frac{t^{p^{s}}}{x_{1}\cdot\cdot\cdot x_{n}})  & \mathrm{for}\ 2\leq l\leq n.
		\end{array}\right.
	\end{equation}
	Observe that $\overline{D}_{t^{p^{s}},l}^{(1)}=D_{t^{p^{s}},l}^{(1)}=D_{t^{p^{s}},l}\ \mathrm{mod}\ \gamma^{1/mq}$.
	Hence we can define the reduced complex $\Omega^{\bullet}(\overline{\mathcal{C}}_{p^{s}},\overline{D}_{t^{p^{s}}}^{(1)})$ in the same way as $\Omega^{\bullet}(\mathcal{C}_{p^{s}},D_{t^{p^{s}}})$, and denote its cohomology by $H^{\bullet}(\overline{\mathcal{C}}_{p^{s}},\overline{D}_{t^{p^{s}}}^{(1)})$. It follows that the reduction $\Omega^{\bullet}(\mathcal{C}_{p^{s}},D_{t^{p^{s}}})$ mod $\gamma^{1/mq}$ is isomorphic to $\Omega^{\bullet}(\overline{\mathcal{C}}_{p^{s}},\overline{D}_{t^{p^{s}}}^{(1)})$ as $\overline{\mathcal{O}}_{p^{s}}$\text{-}modules. 
	
	For $r\in \mathbb{R}$, denote by $\lfloor r\rfloor$  the floor of $r$, and $\lceil r\rceil$ the ceiling of $r$.
	
	\begin{lem}\label{lem24}
		For $\xi\in\mathrm{Fil}^{N}\mathcal{C}_{p^{s}}$ and $l\leq \min\{\lfloor mN\rfloor,mn\}$,there exist $a_{l}\in\mathrm{Fil}^{N-l/m}\mathcal{O}_{p^{s}}$ and $\zeta_{j}\in\mathrm{Fil}^{N-1}\mathcal{C}_{p^{s}}$ such that
		$$\xi=\sum_{l=0}^{\min\{\lfloor mN\rfloor,mn\}}a_{l}e_{l}+\sum_{j=1}^{n}D_{t^{p^{s}},j}^{(1)}(\zeta_{j}).$$
	\end{lem}
	\begin{proof}
		It follows from \cite[Theorem 2.5]{HS172} that the complex $\Omega^{\bullet}(\overline{\mathcal{C}}_{p^{s}},\overline{D}_{t^{p^{s}}}^{(1)})$ is acyclic except in the top dimension $n$ where $H^{n}(\overline{\mathcal{C}}_{p^{s}},\overline{D}_{t^{p^{s}}}^{(1)})$ is a free $\overline{\mathcal{O}}_{p^{s}}$\text{-}module of rank $mn+1$ with the basis $\{\bar{e}_{i}\}_{0\leq i\leq mn}$ defined after (\ref{basis}).
		
		Notice that
		$$x_{1}^{m}-\frac{t^{p^{s}}}{x_{1}\cdot\cdot\cdot x_{n}},\  x_{l}-\frac{t^{p^{s}}}{x_{1}\cdot\cdot\cdot x_{n}}\in\mathrm{Fil}^{1}\bar{\mathcal{C}}_{p^{s}}\ \mathrm{with}\ 2\leq l\leq n,$$
		and $x_{i}\frac{\partial}{\partial x_{i}}$ preserve weights. By Lemma \ref{Lemma 2.5} we have that $\overline{D}_{t^{p^{s}},i}^{(1)}\mathrm{Fil}^{N}\overline{\mathcal{C}}_{p^{s}}\subset\mathrm{Fil}^{N+1}\overline{\mathcal{C}}_{p^{s}}$ for all $1\leq i\leq n$.
		Hence for $\bar{\xi}\in\mathrm{Fil}^{N}\overline{\mathcal{C}}_{p^{s}}$, there exist $\bar{a}_{l}\in\mathrm{Fil}^{N-l/m}\overline{\mathcal{O}}_{p^{s}}$ and $\bar{\zeta}_{j}\in\mathrm{Fil}^{N-1}\overline{\mathcal{C}}_{p^{s}}$ such that
		$$\bar{\xi}=\sum_{l=0}^{\min\{\lfloor mN\rfloor,mn\}}\bar{a}_{l}\bar{e}_{l}+\sum_{j=1}^{n}\overline{D}_{t^{p^{s}},j}^{(1)}(\bar{\zeta}_{j}).$$
		Then Lemma \ref{lem24} follows from a standard lifting argument.
	\end{proof}

	To recover a $p$-adic estimate from $D^{(1)}_{t^{p^{s}}}$ to $D_{t^{p^{s}}}$, we need the following lemma.
	\begin{lem} \label{Lemma 2.7}
		For $\xi\in\mathrm{Fil}^{N}\mathcal{C}_{p^{s}}$, there exist $a_{l}\in\mathrm{Fil}^{N-l/m}\mathcal{O}_{p^{s}}$, $\zeta_{j}\in\mathrm{Fil}^{N-1}\mathcal{C}_{p^{s}}$ and $\varpi_{r}\in\mathrm{Fil}^{N+p^{r}-1}\mathcal{C}_{p^{s}}$ such that
		$$\xi=\sum_{l=0}^{\min\{\lfloor mN\rfloor,mn\}}a_{l}e_{l}+\sum_{j=1}^{n}D_{t^{p^{s}},j}(\zeta_{j})+\sum_{r=1}^{\infty}p^{p^{r}-1}\varpi_{r}.$$
	\end{lem}
	\begin{proof}
		For $\xi\in\mathrm{Fil}^{N}\mathcal{C}_{p^{s}}$, from Lemma \ref{lem24} we have that
		\begin{equation}\label{xi}
			\xi=\sum_{l=0}^{\min\{\lfloor mN\rfloor,mn\}}a_{l}e_{l}+\sum_{j=1}^{n}D_{t^{p^{s}},j}^{(1)}(\zeta_{j})
		\end{equation}
		with $\zeta_j\in {\rm Fil}^{N-1}\mathcal{C}_{p^s}$.
		
		It follows from (\ref{differential}) and (\ref{differential (1)}) that
		$$D^{(1)}_{t^{p^{s}},1}=D_{t^{p^{s}},1}-\sum_{r=1}^{\infty}\gamma_{r}p^{r}(mx_{1}^{mp^{r}}-\frac{t^{p^{s+r}}}{x_{1}^{p^{r}}\cdot\cdot\cdot x_{n}^{p^{r}}}).
		$$
		and
		$$D^{(1)}_{t^{p^{s}},l}=D_{t^{p^{s}},l}-\sum_{r=1}^{\infty}\gamma_{r}p^{r}(x_{l}^{p^{r}}-\frac{t^{p^{s+r}}}{x_{1}^{p^{r}}\cdot\cdot\cdot x_{n}^{p^{r}}})
		$$
		for $2\le l\le n$.
		Then by \hyperref[(2.4)]{(2.4)} we can write $\gamma_{r}p^{r}=p^{p^{r}-1}\gamma^{p^{r}}\tau_{r}$ where $\tau_{r}$ is a $p$-adic unit. Let
		$$\eta_{p^{r},1}:=-\tau_{r}\gamma^{p^{r}}(mx_{1}^{mp^{r}}-\frac{t^{p^{s+r}}}{x_{1}^{p^{r}}\cdot\cdot\cdot x_{n}^{p^{r}}}) $$
		and
		$$\eta_{p^{r},l}:=-\tau_{r}\gamma^{p^{r}}(x_{l}^{p^{r}}-\frac{t^{p^{s+r}}}{x_{1}^{p^{r}}\cdot\cdot\cdot x_{n}^{p^{r}}}) $$
		for $2\le l\le n$.
		Hence (\ref{xi}) can be written as
		$$\xi=\sum_{l=0}^{\min\{\lfloor mN\rfloor,mn\}}a_{l}e_{l}+\sum_{j=1}^{n}D_{t^{p^{s}},j}(\zeta_{j})+\sum_{r=1}^{\infty}\sum_{j=1}^{n}p^{p^{r}-1}\eta_{p^{r},j}\zeta_{j},$$
		where $\eta_{p^{r},j}\zeta_{j}\in\mathrm{Fil}^{p^{r}+N-1}\mathcal{C}_{p^{s}}$ for $1\leq j\leq n$. The result then follows by setting $\varpi_{r}:=\sum_{j=1}^{n}\eta_{p^{r},j}\zeta_{j}$.
	\end{proof}

	In order to detect finer $p$-adic estimates of coefficients, for $b,c\in\mathbb{R}$, we define
	\begin{equation}
		\mathcal{O}_{p^{s}}(b;c):=\{\sum_{r=0}^{\infty}\zeta(r)\gamma^{\frac{(mn+1)r}{mp^{s}}}t^{r}: \zeta(r)\in\mathfrak{D}_q,\ \ord_{p}\zeta(r)\geq br+c\}.
	\end{equation}
	And to simplify the notations for following properties, for any positive integer $i$, denote $$\mathcal{F}_{i}=\{\textbf{a}=(a_{1},a_{2},\cdot\cdot\cdot)\in\mathbb{Z}^{^{\infty}}_{\geq 0}: a_{j}\neq 0\ \mathrm{for\ all}\ 1\leq j\leq i,\ a_{j}=0\ \mathrm{for}\ \mathrm{all}\ j\geq i+1\}$$
	and let $\rho(\textbf{a}):=p^{a_{1}}+\cdot\cdot\cdot +p^{a_{i}}-i$ for $\textbf{a}\in\mathcal{F}_{i}$.
	
	\begin{thm} \label{Theorem 2.8}
		For $\xi\in\mathrm{Fil}^{N}\mathcal{C}_{p^{s}}$, there exist $C(l,\xi)\in\mathcal{O}_{p^{s}}(\frac{mn+1}{mp^{s}};\frac{l}{m}-N)$ such that
		$$\xi=\sum_{l=0}^{mn}C(l,\xi)e_{l}\ \ \mathrm{mod}\ D_{t^{p^{s}}}$$
		in $H^{n}(\mathcal{C}_{p^{s}},D_{t^{p^{s}}})$. The coefficient $C(l,\xi)$ can be explicitly expressed as
		$$\sum_{l=0}^{mn}C(l,\xi)e_{l}=\sum_{l=0}^{\min\{\lfloor mN\rfloor,mn\}}a(l,\xi)e_{l}+\sum_{i=1}^{\infty}\sum_{\textbf{r}\in\mathcal{F}_{i}}\sum_{l=0}^{\min\{mn,m\rho(\textbf{r})+\lfloor mN\rfloor\}}p^{\rho(\textbf{r})}a(l,\xi;\textbf{r})e_{l},$$
		where $a(l,\xi)\in\mathrm{Fil}^{N-l/m}\mathcal{O}_{p^{s}}$ and $a(l,\xi;\textbf{r})\in\mathrm{Fil}^{\rho(\textbf{r})+N-l/m}\mathcal{O}_{p^{s}}$.
	\end{thm}
	
	\begin{proof}
		For any positive integer $L$, we claim that
		$\xi$ can be written as follows
		\begin{equation} \label{(2.13)}
			\begin{aligned}
				\xi=&\sum_{l=0}^{\min\{\lfloor mN\rfloor,mn\}}a(l,\xi)e_{l}+\sum_{i=1}^{L}\sum_{\textbf{r}\in\mathcal{F}_{i}}\sum_{l=0}^{\min\{mn,\lfloor mN\rfloor+m\rho(\textbf{r})\}}p^{\rho(\textbf{r})}a(l,\xi;\textbf{r})e_{l}\\
				&+\sum_{j=1}^{n}D_{t^{p^{s}},j}\big(\zeta_{j}+\sum_{i=1}^{L}\sum_{\textbf{r}\in\mathcal{F}_{i}}p^{\rho(\textbf{r})}\zeta(j;\textbf{r})\big)+\sum_{\textbf{r}\in\mathcal{F}_{L+1}}p^{\rho(\textbf{r})}\varpi_{\textbf{r}}
			\end{aligned}
		\end{equation}
		where the coefficients $a(l,\xi)\in\mathrm{Fil}^{N-l/m}\mathcal{O}_{p^{s}}$, $a(l,\xi;\textbf{r})\in\mathrm{Fil}^{\rho(\textbf{r})+N-l/m}\mathcal{O}_{p^{s}}$, $\xi_{j}\in\mathrm{Fil}^{N-1}\mathcal{C}_{p^{s}}$, $\xi(j;\textbf{r})\in\mathrm{Fil}^{N+\rho(\textbf{r})-1}\mathcal{C}_{p^{s}}$ and $\varpi_{\textbf{r}}\in\mathrm{Fil}^{N+\rho(\textbf{r})}\mathcal{C}_{p^{s}}$.
		
		We prove the claim by induction on $L$.
		By Lemma \ref{Lemma 2.7}, there exist $a(l,\xi)\in\mathrm{Fil}^{N-l/m}\mathcal{O}_{p^{s}}$, $\zeta_{j}\in\mathrm{Fil}^{N-1}\mathcal{C}_{p^{s}}$ and $\varpi_{r_{1}}\in\mathrm{Fil}^{N+p^{r_{1}}-1}\mathcal{C}_{p^{s}}$ such that
		\begin{equation} \label{(2.14)}
			\xi=\sum_{l=0}^{\min\{\lfloor mN\rfloor,mn\}}a(l,\xi)e_{l}+\sum_{j=1}^{n}D_{t^{p^{s}},j}(\zeta_{j})+\sum_{r_{1}=1}^{\infty}p^{p^{r_{1}}-1}\varpi_{r_{1}}.
		\end{equation}
		Then we apply Lemma \ref{Lemma 2.7} to $\varpi_{r_{1}}$ and  obtain
		\begin{equation} \label{(2.15)}
			\varpi_{r_{1}}=\sum_{l=0}^{\min\{\lfloor mN\rfloor+mp^{r_{1}}-m,mn\}}a(l,\xi;r_{1})e_{l}+\sum_{j=1}^{n}D_{t^{p^{s}},j}(\zeta(j;r_{1}))+\sum_{r_{2}=1}^{\infty}p^{p^{r_{2}}-1}\varpi_{r_{1},r_{2}}
		\end{equation}
		for some $a(l,\xi;r_{1})\in\Fil^{N+p^{r_{1}}-1-l/m}\mathcal{O}_{p^{s}}$, $\zeta(j;r_{1})\in\Fil^{N+p^{r_{1}}-2}\mathcal{C}_{p^{s}}$ and $\varpi_{r_{1},r_{2}}\in\Fil^{N+p^{r_{1}}+p^{r_{2}}-2}\mathcal{C}_{p^{s}}$.
		Let $\mathbf{r}_{1}:=(r_{1},0,\cdot\cdot\cdot)\in\mathcal{F}_{1}$ and $\mathbf{r}_{2}:=(r_{1},r_{2},0,\cdot\cdot\cdot)\in\mathcal{F}_{2}$.
		Then combining  (\ref{(2.14)}) and (\ref{(2.15)}),
		we get that
		\begin{align*}
			\xi=&\sum_{l=0}^{\min\{\lfloor mN\rfloor,mn\}}a(l,\xi)e_{l}+\sum_{l=0}^{\min\{\lfloor mN\rfloor+m\rho(\mathbf{r}_1),mn\}}p^{\rho(\mathbf{r}_1)}a(l,\xi;\mathbf{r}_{1})e_{l}\\
			&+\sum_{j=1}^{n}D_{t^{p^{s}},j}(\zeta_{j}+\sum_{\mathbf{r}_{1}\in\mathcal{F}_1}p^{\rho(\mathbf{r}_1)}\zeta(j;\mathbf{r}_1))+\sum_{\mathbf{r}_2\in \mathcal{F}_2}p^{\rho(\mathbf{r}_2)}\varpi_{\textbf{r}_2}.
		\end{align*}
		Hence the claim holds for $L=1$.
		
		Suppose the claim is true for $L$. That is, $\xi$ can be written as (\ref{(2.13)}).
		Then apply Lemma \ref{Lemma 2.7} for $\varpi_{\textbf{r}}\in\Fil^{N+\rho(\textbf{r})}\mathcal{C}_{p^{s}}$ to obtain that
		$$\varpi_{\textbf{r}}=\sum_{l=0}^{\min\{\lfloor mN\rfloor+m\rho(\textbf{r}),mn\}}a(l,\xi;\textbf{r})e_{l}+\sum_{j=1}^{n}D_{t^{p^{s}},j}(\zeta(j;\textbf{r}))+\sum_{r_{L+2}=1}^{\infty}p^{p^{r_{L+2}}-1}\varpi_{\textbf{r},r_{L+2}}$$
		where $a(l,\xi;\textbf{r})\in\Fil^{N+\rho(\textbf{r})-l/m}\mathcal{O}_{p^{s}}$, $\zeta(j;\textbf{r})\in\Fil^{N+\rho(\textbf{r})-1}\mathcal{C}_{p^{s}}$ and $\varpi_{\textbf{r},r_{L+2}}\in\Fil^{N+\rho(\textbf{r})+p^{r_{L+2}}-1}\mathcal{C}_{p^{s}}$. Note that $\mathbf{r}\in \mathcal{F}_{L+1}$. Then we set $\mathbf{r}_{L+2}:=(r_{1},\cdot\cdot\cdot,r_{L+2},0,\cdot\cdot\cdot)\in\mathcal{F}_{L+2}$. Then by the induction we conclude the claim holds.
		
		If we let $L\rightarrow\infty$ in (\ref{(2.13)}), then the term $\sum_{\textbf{r}\in\mathcal{F}_{L+1}}p^{\rho(\textbf{r})}\varpi_{\textbf{r}}\rightarrow 0$.
		It remains to show that  $C(l,\xi)\in\mathcal{O}_{p^{s}}(\frac{mn+1}{mp^{s}};\frac{l}{m}-N)$.
		Since
		$a(l,\xi)\in\Fil^{N-l/m}\mathcal{O}_{p^{s}}$, we can write
		$$a(l,\xi)=\sum_{\begin{subarray}{c}
				r\geq 0, \
				\frac{(mn+1)r}{mp^{s}}\leq N-\frac{l}{m}
		\end{subarray}}A(l,\xi,r)\gamma^{\frac{(mn+1)r}{mp^{s}}}t^{r},$$
		where the coefficients $A(l,\xi,r)\in\mathfrak{D}_q$.
		Then $\ord_{p}A(l,\xi,r)\geq 0\geq \frac{(mn+1)r}{mp^{s}}+\frac{l}{m}-N$.
		This shows $a(l,\xi)\in\mathcal{O}_{p^{s}}(\frac{mn+1}{mp^{s}};\frac{l}{m}-N)$.
		
		For $a(l,\xi;\textbf{r})\in\Fil^{\rho(\textbf{r})+N-l/m}\mathcal{O}_{p^{s}}$, we write
		$$p^{\rho(\textbf{r})}a(l,\xi;\textbf{r})=\sum_{r\geq0,\ \frac{(mn+1)r}{mp^{s}}\leq\rho(\textbf{r})+N-\frac{l}{m}}p^{\rho(\textbf{r})}A(l,\xi;\textbf{r},r)\gamma^{\frac{(mn+1)r}{mp^{s}}}t^{r}$$
		with the coefficients $A(l,\xi;\textbf{r},r)\in\mathfrak{D}_q$. We have
		$$\ord_{p}p^{\rho(\textbf{r})}A(l,\xi;\textbf{r},r)\geq\rho(\textbf{r})\geq\frac{(mn+1)r}{mp^{s}}+\frac{l}{m}-N.$$
		Hence $p^{\rho(\textbf{r})}a(l,\xi;\textbf{r})\in\mathcal{O}_{p^{s}}(\frac{mn+1}{mp^{s}};\frac{l}{m}-N)$. Thus $C(l,\xi)\in\mathcal{O}_{p^{s}}(\frac{mn+1}{mp^{s}};\frac{l}{m}-N)$.  This finishes the proof of Theorem \ref{Theorem 2.8}.
	\end{proof}

	\subsection{Connection map on the relative cohomology}
	Recall that
	$$\nabla_{t^{p^{s}}}=\exp(-H(t^{p^{s}},x))\circ t\frac{d}{dt}\circ\exp(H(t^{p^{s}},x))=t\frac{d}{dt}+\sum_{j=0}^{\infty}\frac{\gamma_{j}p^{s+j}t^{p^{s+j}}}{x_{1}^{p^{j}}x_{2}^{p^{j}}\cdot\cdot\cdot x_{n}^{p^{j}}}$$
	and it induces a connection map on $H^{n}(\mathcal{C}_{p^{s}},D_{t^{p^{s}}})$, which is also  denoted by $\nabla_{t^{p^{s}}}$.

	However, the infinite series is the main obstacle to write down the explicit connection matrix. We then turn to the connection map on the reduced relative cohomology, we focus on the case $s=0$ since the terms contains positive power of $p$ vanish in the residue field $\mathbb{F}_q$. 
	Define
	$$\overline{\nabla}^{(1)}=t\frac{d}{dt}+\frac{t}{x_{1}\cdot\cdot\cdot x_{n}},$$
	and we denote $\overline{G}=t/(x_{1}\cdot\cdot\cdot x_{n})$. Observe that $\overline{\nabla}^{(1)}\equiv\nabla\ \mathrm{mod}\ \gamma^{1/mq}$.
	Since $\overline{\nabla}^{(1)}$ commutes with $\overline{D}_{t,l}^{(1)}$ for $1\leq l\leq n$, with an abuse use of notation, it induces a connection map $\overline{\nabla}^{(1)}$ on $H^{n}(\overline{\mathcal{C}},\overline{D}_{t}^{(1)})$. Since the basis $\{\bar{e}_{i}\}_{0\leq i\leq mn}$ is independent of the parameter $t$, we have that $\overline{\nabla}^{(1)}(\bar{e}_{i})=\overline{G}(\bar{e}_{i})$ for $0\leq i\leq mn$. For each $0\leq i\leq mn$, we write $i=\varepsilon m+j$. A simple calculation shows that on the cohomology space $H^{n} (\overline{\mathcal{C}},\overline{D}_{t}^{(1)})=\overline{\mathcal{C}}/\sum_{l=1}^{n} \overline{D}^{(1)}_{t,l}\overline{\mathcal{C}}$
	we have that
	\begin{equation}  \label{(2.18)}
		\overline{\nabla}^{(1)}(\bar{e}_{i})=\overline{G}(\bar{e}_{i})=\overline{G}(\bar{e}_{\varepsilon m+j})=\left\{\begin{array}{lll}
			m\bar{e}_{m}  & \ \mathrm{for}\ i=0 \\
			\bar{e}_{(\varepsilon+1)m+j}  &\ \mathrm{for}\ 0\leq\varepsilon\leq n-2,1\leq j\leq m   \\
			t\bar{e}_{j-1}  &\ \mathrm{for}\ \varepsilon=n-1,\ \ 1\leq j\leq m
		\end{array}\right.  .
	\end{equation}
	With respect to this basis, the matrix of $\overline{G}$ is
	\begin{equation}
		\overline{G}\left(\begin{array}{c}
			\bar{e}_0 \\ \bar{e}_1 \\ \vdots \\ \bar{e}_{mn}
		\end{array}\right)=\left(\begin{array}{cc}
			\mathbf{0} &   B_{m,n}\\
			t\cdot I_{m} & \mathbf{0}
		\end{array}\right)\left(\begin{array}{c}
			\bar{e}_0 \\ \bar{e}_1 \\ \vdots \\ \bar{e}_{mn}
		\end{array}\right)
	\end{equation}
	where $I_{m}$ is the $m\times m$ identity matrix, $B_{m,n}=\mathrm{diag}\{m,1,1,\cdot\cdot\cdot,1\}$ is a $(mn-m+1)\times(mn-m+1)$ diagonal matrix. 
	
	Motivated by above calculation, we define
	\begin{equation}
		\nabla^{(1)}=t\frac{d}{dt}+G_{0}
	\end{equation}
	where the matrix $G_{0}$ linearly acts on $H^n (\mathcal{C},D_t)$ with respect to the normalized basis $\{e_{i}\}_{0\leq i\leq mn}$ via 
	\begin{equation*}
		\nabla^{(1)}(e_{i})=G_{0}(e_i)=G_{0}(e_{\varepsilon m+j})=\left\{\begin{array}{lll}
			m e_{m}  & \ \mathrm{for}\ i=0 \\
			e_{(\varepsilon+1)m+j}  &\ \mathrm{for}\ 0\leq\varepsilon\leq n-2,1\leq j\leq m   \\
			\gamma^{\frac{mn+1}{m}}t e_{j-1}  &\ \mathrm{for}\ \varepsilon=n-1,\ \ 1\leq j\leq m
		\end{array}\right.  .
	\end{equation*}
	One still observe that after the reduction $\overline{\nabla}^{(1)}=\nabla \mathrm{mod}\ \gamma^{1/mq}=\nabla^{(1)} \mathrm{mod}\ \gamma^{1/mq}$, and we then have the $p$-adic estimate for the error terms in $\nabla-\nabla^{(1)}$:
	
	\begin{lem} \label{Lemma 2.11}
		For every $0\leq l\leq mn$, in $H^{n}(\mathcal{C},D_{t})$ we can express
		$$\nabla^{(1)}(e_{l})=\nabla(e_{l})+\sum_{i=1}^{\infty}\sum_{\textbf{r}\in\mathcal{F}_{i}}\sum_{j=0}^{\min\{m\rho(\textbf{r})+l+m,mn\}}p^{\rho(\textbf{r})}b(j,l;\textbf{r})e_{j}\ \mathrm{mod}\ D_{t}$$
		with $b(j,l;\textbf{r})\in\Fil^{\rho(\textbf{r})+1+\frac{l-j}{m}}\mathcal{O}$.
	\end{lem}
	\begin{proof}
		We prove this lemma by investigating each individual basis element $e_i$. Firstly recall 
		\begin{equation*}  \label{differential}
			D_{t,l}=\left\{\begin{array}{cc}
				x_{1}\frac{\partial}{\partial x_{1}}+\sum\limits_{j=0}^{\infty}\gamma_{j}p^{j}(mx_{1}^{mp^{j}}-\frac{t^{p^{j}}}{x_{1}^{p^{j}}\cdot\cdot\cdot x_{n}^{p^{j}}})   &  \mathrm{for}\ l=1\\
				x_{l}\frac{\partial}{\partial x_{l}}+\sum\limits_{j=0}^{\infty}\gamma_{j}p^{j}(x_{l}^{p^{j}}-\frac{t^{p^{j}}}{x_{1}^{p^{j}}\cdot\cdot\cdot x_{n}^{p^{j}}})  &  \mathrm{for}\ 2\leq l\leq n.
			\end{array}\right.
		\end{equation*}
		Case I. For $e_0$, $D_{t,1}(e_0)=D_{t,1}(1)=0$ in $H^{n}(\mathcal{C},D_t)$ shows
		$$\gamma mx_{1}^{m}-\frac{\gamma t}{x_1\cdots x_n}+\sum\limits_{j=1}^{\infty}\gamma_{j}p^{j}(mx_{1}^{mp^{j}}-\frac{t^{p^{j}}}{x_{1}^{p^{j}}\cdot\cdot\cdot x_{n}^{p^{j}}})=0.$$
		Hence 
		\begin{equation*}
			\begin{aligned}
				\nabla(e_0)&=\frac{\gamma t}{x_1 \cdots x_n}+\sum_{j=1}^{\infty}\frac{\gamma_{j}p^{j}t^{p^j}}{x_{1}^{p^j}\cdots x_{n}^{p^j}} \\
				&=\gamma mx_{1}^{m}+\sum\limits_{j=1}^{\infty}\gamma_{j}p^{j}(mx_{1}^{mp^{j}}-\frac{t^{p^{j}}}{x_{1}^{p^{j}}\cdot\cdot\cdot x_{n}^{p^{j}}})+\sum_{j=1}^{\infty}\frac{\gamma_{j}p^{j}t^{p^j}}{x_{1}^{p^j}\cdots x_{n}^{p^j}} \\
				&=\nabla^{(1)}(e_0)+\sum_{j=1}^{\infty}\gamma_{j}p^{j}mx_{1}^{mp^j}.
			\end{aligned}    
		\end{equation*}
		Case II. For $e_{\varepsilon m+\alpha}$, $0\leq \varepsilon\leq n-2$, $1\leq \alpha\leq m$, the relation $D_{t,\varepsilon+2}(x_{1}^{\alpha}x_{2}\cdots x_{\varepsilon+1})=0$ gives
		$$\gamma x_{1}^{\alpha}x_{2}\cdots x_{\varepsilon+2}-\frac{\gamma tx_{1}^{\alpha}x_{2}\cdots x_{\varepsilon+1}}{x_{1}\cdots x_{n}}+x_{1}^{\alpha}x_{2}\cdots x_{\varepsilon+1}\sum_{j=1}^{\infty}\gamma_{j}p^{j}(x_{\varepsilon+2}^{p^{j}}-\frac{t^{p^{j}}}{x_{1}^{p^{j}}\cdot\cdot\cdot x_{n}^{p^{j}}})=0.$$
		Hence
		\begin{equation*}
			\begin{aligned}
				\nabla(e_{\varepsilon m+\alpha})&=\nabla(\gamma^{\frac{\varepsilon m+\alpha}{m}}x_{1}^{\alpha}x_{2}\cdots x_{\varepsilon+1})=\frac{\gamma^{\frac{(\varepsilon+1)m+\alpha}{m}}tx_{1}^{\alpha}x_{2}\cdots x_{\varepsilon+1}}{x_{1}\cdots x_{n}}+\gamma^{\frac{\varepsilon m+\alpha}{m}}x_{1}^{\alpha}x_{2}\cdots x_{\varepsilon+1}\sum_{j=1}^{\infty}\frac{\gamma_{j}p^{j}t^{p^j}}{x_{1}^{p^j}\cdots x_{n}^{p^j}}   \\
				&=e_{(\varepsilon+1)m+\alpha}+e_{\varepsilon m+\alpha}\sum_{j=1}^{\infty}\gamma_{j}p^{j}(x_{\varepsilon+2}^{p^j}-\frac{t^{p^j}}{x_{1}^{p^j}\cdots x_{n}^{p^j}})+e_{\varepsilon m+\alpha}\sum_{j=1}^{\infty}\frac{\gamma_{j}p^{j}t^{p^j}}{x_{1}^{p^j}\cdots x_{n}^{p^j}} \\
				&=\nabla^{(1)}(e_{\varepsilon m+\alpha})+e_{\varepsilon m+\alpha}\sum_{j=1}^{\infty}\gamma_{j}p^{j}x_{\varepsilon+2}^{p^j}.
			\end{aligned}    
		\end{equation*}
		Case III. For $e_{(n-1)m+\alpha}$, $1\leq\alpha\leq m$, we directly compute
		\begin{equation*}
			\begin{aligned}
				\nabla(e_{(n-1)m+\alpha})&=\nabla(\gamma^{\frac{(n-1)m+\alpha}{m}}x_{1}^{\alpha}x_{2}\cdots x_{n})=\gamma^{\frac{nm+\alpha}{m}}tx_{1}^{\alpha-1}+e_{(n-1)m+\alpha}\sum_{j=1}^{\infty}\frac{\gamma_{j}p^{j}t^{p^j}}{x_{1}^{p^j}\cdots x_{n}^{p^j}}  \\  
				&=\nabla^{(1)}(e_{(n-1)m+\alpha})+e_{(n-1)m+\alpha}\sum_{j=1}^{\infty}\frac{\gamma_{j}p^{j}t^{p^j}}{x_{1}^{p^j}\cdots x_{n}^{p^j}} 
			\end{aligned}    
		\end{equation*}

		By (\ref{(2.4)}), we may write $\gamma_{r}p^{r}=p^{p^{r}-1}\gamma^{p^r}\tau_r$ where $\tau_r$ is a $p$-adic unit. Combining above three cases we let
		\begin{equation*}
			\nu_{r,l}=\begin{cases}
				-\gamma^{p^r}\tau_{r}mx_{1}^{mp^r } & \mathrm{when}\ l=0, \\
				-\gamma^{p^{r}}\tau_{r}x_{\varepsilon+2}^{p^r} & \mathrm{when}\ l=\varepsilon m+\alpha, \ 0\leq\varepsilon\leq n-2, \ 1\leq\alpha\leq m, \\
				-\gamma^{p^r}\tau_{r}\frac{t^{p^r}}{x_{1}^{p^r}\cdots x_{n}^{p^r}} & \mathrm{when}\ l=(n-1)m+\alpha, \ 1\leq\alpha\leq m .
			\end{cases}    
		\end{equation*}
		
		Then
		\begin{equation} \label{(2.16)}
			\nabla^{(1)}(e_{l})=\nabla(e_{l})+\sum_{r=1}^{\infty}p^{p^{r}-1}\nu_{r,l}e_{l}.
		\end{equation}
		Note that $\nu_{r,l}e_{l}\in\Fil^{p^{r}+l/m}\mathcal{C}$.
		Applying Theorem \ref{Theorem 2.8}, we have
		\begin{equation} \label{(2.17)}
			\begin{aligned}
				p^{p^r-1}\nu_{r,l}e_{l}&=\sum_{j=0}^{\min\{mp^{r}+l,mn\}}p^{p^r-1}a(j,r,l)e_{j} \\ &+\sum_{i=1}^{\infty}\sum_{\textbf{r}\in\mathcal{F}_{i}}\sum_{j=0}^{\min\{m\rho(\textbf{r})+l+mp^{r},mn\}}p^{\rho(\textbf{r})+p^r-1}a(j,r,l;\textbf{r})e_{j}\ \ \mathrm{mod}\ D_{t},
			\end{aligned}
		\end{equation}
		where $a(j,r,l)\in\Fil^{p^{r}+\frac{l-j}{m}}\mathcal{O}$ and $a(j,r,l;\textbf{r})\in\Fil^{\rho(\textbf{r})+p^{r}+\frac{l-j}{m}}\mathcal{O}$. Let
		\begin{equation*}
			b(j,l;\textbf{r})=\begin{cases}
				a(j,r_i ,l;(r_1, \cdots,r_{i-1},0\cdots))  & \mathrm{when}\ \textbf{r}=(r_1 ,\cdots ,r_i ,0 ,\cdots)\in\mathcal{F}_{i}\ \mathrm{for}\ i>1  ,   \\
				a(j,r,l) & \mathrm{when}\ \textbf{r}=(r,0,\cdots)\in\mathcal{F}_1 .
			\end{cases}    
		\end{equation*}
		Combining (\ref{(2.16)}) and (\ref{(2.17)}), we conclude the Lemma \ref{Lemma 2.11} holds by this labeling.
	\end{proof}
	
	\begin{rmk}
		One immediately note that the relative cohomology constructed in Section 2 of this weighted family $Kl_{n,m}$ is a one-parameter pullback of a $p$-adic GKZ hypergeometric $F$-isocrystal with trivial multiplicative character, which has been studied in details in \cite{FW23}. Concretely, for the connection equation, if one replace $\gamma$ in all above constructions starting from Section 2 by the classical Dwork's $\pi\in\overline{\mathbb{Q}}_p$ where $\pi^{p-1}=-p$, and replace the construction $H(t,x)$ in (\ref{H}) by $H(t,x)=\pi Kl_{n,m}(t,x)$, then one readily compute that
		$$\nabla=\nabla^{(1)}=t\frac{d}{dt}+\frac{\pi t}{x_1\cdots x_n}.$$
		
		Like above, now $\nabla$ acts on $H^{n}(\mathcal{C},D_{t})$, with respect to the normalized basis $\{e_i\}_{0\leq i\leq nm}$, the corresponding matrix $G_0$ becomes 
		\begin{equation}
			G_0=\left(\begin{array}{cc}
				\mathbf{0} &   B_{m,n}\\
				\pi^{\frac{mn+1}{m}} t\cdot I_{m} & \mathbf{0}
			\end{array}\right)
		\end{equation}
		with scalar equation of $t\frac{d}{dt}-G_0=0$ to be 
		\begin{equation*}
			[t\frac{d}{dt}\prod_{j=0}^{m-1}(t\frac{d}{dt}-j)^{n}-m\pi^{mn+1}t^{m}]y=0.   
		\end{equation*}
		A solution of this scalar equation near $t=0$ is just the generalized hypergeometric series
		\begin{equation*}
			y(t)=
			{}_0F_{mn}\!\left(
			\begin{matrix}
				\\
				\underbrace{\frac{1}{m},\ldots,\frac{1}{m}}_{n},
				\underbrace{\frac{2}{m},\ldots,\frac{2}{m}}_{n},
				\ldots,
				\underbrace{1,\ldots,1}_{n}
			\end{matrix}
			\,;\,
			\frac{\pi^{mn+1}t^m}{m^{mn}}
			\right).
		\end{equation*}
		When $m=1$, this becomes the familiar hyper-Kloosterman connection. Moreover, let 
		\begin{equation*}
			A=
			\begin{pmatrix}
				m & 0 & \cdots & 0 & -1\\
				0 & 1 & \cdots & 0 & -1\\
				\vdots & \vdots & \ddots & \vdots & \vdots\\
				0 & 0 & \cdots & 1 & -1
			\end{pmatrix}
			\in M_{n\times(n+1)}(\mathbb Z),
		\end{equation*}
		with the GKZ twist parameters $\gamma_{1}=\cdots=\gamma_{n}=0$ in [\cite{FW23}, Section 1.1], let $F(x_{1},\cdots, x_{n+1})$ be the solution of the corresponding $A$-hypergeometric system, one will realize the solution $y(t)$ as
		$$y(t)=F(\pi,\cdots,\pi,\pi t).$$
		Thus the connection for $Kl_{n,m}$ is just a codimension 1 instance of the $p$-adic GKZ system.
	\end{rmk}

	\subsection{ Reduced symmetric-power cohomology} \label{Subsection 3.2}
	
	Let $\overline{\mathcal{S}}_{k}:={\rm Sym}^{k}_{\overline{\mathcal{O}}}H^{n}(\overline{\mathcal{C}},\overline{D}_{t})$ be the free $\overline{\mathcal{O}}$(=$\mathbb{F}_{q}[t]$)-module of rank $\binom{mn+k}{k}$ with a basis
	$${\rm Sym}^{k}\bar{\mathcal{B}}=\{\bar{e}^{\underline{i}}=\bar{e}_{0}^{i_{0}}\cdot\cdot\cdot \bar{e}_{mn}^{i_{mn}}|\underline{i}=(i_{0},\cdot\cdot\cdot,i_{mn})\in I_{k}\}.$$
	We define the connection map $\overline{\nabla}^{(1)}:\overline{\mathcal{S}}_{k}\rightarrow\overline{\mathcal{S}}_{k}$ by
	$$\overline{\nabla}^{(1)}(\bar{\xi}\bar{e}^{\underline{i}}):=t\frac{d}{dt}(\bar{\xi})\bar{e}^{\underline{i}}+\sum_{j=0}^{mn}i_{j}\bar{\xi}\bar{e}_{0}^{i_{0}}\cdot\cdot\cdot\bar{e}_{j}^{i_{j}-1}\cdot\cdot\cdot e_{mn}^{i_{mn}}\overline{\nabla}^{(1)}(\bar{e}_{j}),$$
	and define a $\mathbb{F}_{q}[t]$-linear map $\phi_{\overline{G}}:\overline{\mathcal{S}}_{k}\rightarrow\overline{\mathcal{S}}_{k}$ by
	$$\phi_{\overline{G}}(\bar{\xi}\bar{e}^{\underline{i}}):=\sum_{j=0}^{mn}i_{j}\bar{\xi}\bar{e}_{0}^{i_{0}}\cdot\cdot\cdot\bar{e}_{j}^{i_{j}-1}\cdot\cdot\cdot e_{mn}^{i_{mn}}\overline{G}(\bar{e}_{j}).$$
	Recall we have that $\overline{\nabla}^{(1)}(\bar{e}_{i})=\overline{G}(\bar{e}_{i})$ for $0\leq i\leq mn$, so $\overline{\nabla}^{(1)}=t\frac{d}{dt}+\phi_{\overline{G}}$. We define a complex $\Omega^{\bullet}(\overline{\mathcal{S}}_{k},\overline{\nabla}^{(1)})$ by
	$$\Omega^{0}(\overline{\mathcal{S}}_{k},\overline{\nabla}^{(1)})=\overline{\mathcal{S}}_{k},\ \mathrm{and}\ \Omega^{1}(\overline{\mathcal{S}}_{k},\overline{\nabla}^{(1)})=\overline{\mathcal{S}}_{k}\frac{dt}{t}$$
	with the boundary map $\overline{\nabla}^{(1)}(\eta)=\overline{\nabla}^{(1)}(\eta)\frac{dt}{t}$. Denote its cohomology spaces by $H^{i}(\overline{\mathcal{S}}_{k},\overline{\nabla}^{(1)})$ for $i=0,1$. We note that the linear operator $\phi_{\overline{G}}$ is actually the Higgs operator for the logarithmic Higgs field $\phi_{\overline{G}}\frac{dt}{t}$, similarly, we can define the Higgs complex $\Omega^{\bullet}(\overline{\mathcal{S}}_{k},\phi_{\overline{G}})$ and its cohomology $H^{i}(\overline{\mathcal{S}}_{k},\phi_{\overline{G}})$.
	
	Note that $$\det(\lambda I_{mn+1}-\overline{G})=\lambda^{mn+1}-mt^{m}.$$
	Given that $p\nmid m(mn+1)$, all eigenvalues of $\overline{G}$ are $\lambda=m^{1/(mn+1)}t^{m/(mn+1)}\bar{\zeta}_{mn+1}^{i}$ with $0\leq i\leq mn$, where $\bar{\zeta}_{mn+1}$ is a primitive $(mn+1)$-th root of unity in $\overline{\mathbb{F}}_{p}$.
	
	Let $\mathbb{K}:=\overline{\mathbb{F}_{q}}(t^{1/(mn+1)})$. We then can diagonalize the matrix $\overline{G}$ over the extension field $\mathbb{K}$ by $\overline{G}=P\widetilde{G}P^{-1}$, where $$\widetilde{G}=m^{1/(mn+1)}t^{m/(mn+1)}{\rm diag}(1,\bar{\zeta}_{mn+1},\cdot\cdot\cdot,\bar{\zeta}_{mn+1}^{mn})$$ and $P$ is an invertible matrix over $\mathbb{K}$. 
	
	For $(mn+1)\times(mn+1)$ matrix $P$ over $\mathbb{K}$, define ${\rm Sym}^{k}M$ on $\overline{\mathcal{S}}_{k}\otimes\mathbb{K}$ by
	$${\rm Sym}^{k}M(\bar{\xi}\bar{e}^{\underline{i}}):=\bar{\xi}(M\bar{e}_{0})^{i_{0}}\cdot\cdot\cdot (M\bar{e}_{mn})^{i_{mn}},$$
	and define $\phi_{\widetilde{G}}$ similarly to that of $\phi_{\overline{G}}$. Then
	\begin{equation*}
		\begin{aligned}
			\phi_{\overline{G}}(\bar{\xi}\bar{e}^{\underline{i}})&=\sum_{j=0}^{mn}i_{j}\bar{\xi}\bar{e}_{0}^{i_{0}}\cdot\cdot\cdot\bar{e}_{j}^{i_{j}-1}\cdot\cdot\cdot e_{mn}^{i_{mn}}\overline{G}(\bar{e}_{j})    \\
			&=\sum_{j=0}^{mn}i_{j}\bar{\xi}(PP^{-1}\bar{e}_{0})^{i_{0}}\cdot\cdot\cdot(PP^{-1}\bar{e}_{j})^{i_{j}-1}\cdot\cdot\cdot(PP^{-1}\bar{e}_{mn})^{i_{mn}}(P\widetilde{G}P^{-1}\bar{e}_{j}) \\
			&={\rm Sym}^{k}P\circ\phi_{\widetilde{G}}\circ{\rm Sym}^{k}P^{-1}(\bar{\xi}\bar{e}^{\underline{i}}).
		\end{aligned}
	\end{equation*}
	This means $\phi_{\overline{G}}={\rm Sym}^{k}P\circ\phi_{\widetilde{G}}\circ{\rm Sym}^{k}P^{-1}$ as linear operators over $\mathbb{K}$. Since $P$ is an invertible matrix, the injectivity of $\phi_{\bar{G}}$ on $\overline{\mathcal{S}}_{k}$ is equivalent to the injectivity of $\phi_{\widetilde{G}}$ on $\overline{\mathcal{S}}_{k}\otimes\mathbb{K}$.
	
	Recall $$d_{k}(n,m,p):=
	\#\{\underline{i}=(i_{0},\cdot\cdot\cdot,i_{mn})\in I_{k}|\sum_{j=0}^{mn}i_{j}\bar{\zeta}_{mn+1}^{j}= 0\ \mathrm{in}\ \overline{\mathbb{F}}_{p}\}.$$
	
	\begin{thm} \label{d_{k}(n,m,p)=0}
		Suppose $d_{k}(n,m,p)=0$. Then $\phi_{\overline{G}}$ is injective on $\overline{\mathcal{S}}_{k}$. Hence $H^{0}(\overline{\mathcal{S}}_{k},\phi_{\overline{G}})=0$.
	\end{thm}
	\begin{proof}
		Let $\{v_{i}\}_{0\leq i\leq mn}$ be an eigenbasis of $\widetilde{G}$ in $H^{n}(\overline{\mathcal{C}},\overline{D}_{t}^{(1)})\otimes\mathbb{K}$ with eigenvalues $$\{m^{1/(mn+1)}t^{m/(mn+1)}\bar{\zeta}_{mn+1}^{i}\}_{0\leq i\leq mn}.$$ 
		Clearly,  $\overline{\mathcal{S}}_{k}\otimes\mathbb{K}$ admits a basis $\{v^{\underline{i}}=v_{0}^{i_{0}}\cdot\cdot\cdot v_{mn}^{i_{mn}}|\underline{i}\in I_{k}\}$.
		Then we have that
		\begin{equation*}
			\begin{aligned}
				\phi_{\widetilde{G}}(v^{\underline{i}})&=\sum_{j=0}^{mn}i_{j}v_{0}^{i_{0}}\cdot\cdot\cdot v_{j}^{i_{j}-1}\cdot\cdot\cdot v_{mn}^{i_{mn}}(\widetilde{G}v_{j}) \\
				&=\sum_{j=0}^{mn}i_{j}v_{0}^{i_{0}}\cdot\cdot\cdot v_{j}^{i_{j}-1}\cdot\cdot\cdot v_{mn}^{i_{mn}}(m^{\frac{1}{mn+1}}t^{\frac{m}{mn+1}}\bar{\zeta}_{mn+1}^{j}v_{j}) \\
				&=(\sum_{j=0}^{mn}i_{j}\bar{\zeta}_{mn+1}^{j})m^{\frac{1}{mn+1}}t^{\frac{m}{mn+1}}v^{\underline{i}}.
			\end{aligned}
		\end{equation*}
		By $d_{k}(n,m,p)=0$, we get $\phi_{\widetilde{G}}$ is injective.
		Hence $\phi_{\overline{G}}$ is injective.
	\end{proof}
	
	For $\underline{i}=(i_{0},\cdot\cdot\cdot,i_{mn})\in\mathbb{Z}^{mn+1}_{\geq 0}$, we denote $|\underline{i}|:=i_{0}+i_{1}+\cdot\cdot\cdot i_{mn}$, and $$\omega(\underline{i}):=\frac{0}{m}i_{0}+\frac{1}{m}i_{1}+\frac{2}{m}i_{2}+\cdot\cdot\cdot+\frac{mn}{m}i_{mn}.$$ For a monomial $t^{r}\bar{e}^{\underline{i}}\in\overline{\mathcal{S}}_{k}$,  we define the symmetric total weight function for the pair $(r,\underline{i})\in\mathbb{Z}_{\geq 0}\times I_k$ by $$W(r;\underline{i}):=\frac{(mn+1)r}{m}+w(\underline{i}).$$
	Note that this weight function $W$ also has image in $(1/m)\mathbb{Z}_{\geq 0}$. For any $N\in(1/m)\mathbb{Z}_{\geq 0}$ we define a weighted filtration on $\overline{\mathcal{S}}_{k}$ as follows:
	$$\Fil\overline{\mathcal{S}}_{k}^{N}:=\{\mathbb{F}_{q}\text{-}\mathrm{
		vector}\ \mathrm{space}\ \mathrm{generated}\ \mathrm{by}\ t^{r}\bar{e}^{\underline{i}}\ \mathrm{with}\ W(r;\underline{i})\leq N \}$$
	and
	$$\overline{\mathcal{S}}_{k}^{(N)}:=\{\mathbb{F}_{q}\text{-}\mathrm{
		vector}\ \mathrm{space}\ \mathrm{generated}\ \mathrm{by}\ t^{r}\bar{e}^{\underline{i}}\ \mathrm{with}\ W(r;\underline{i})=N \}.$$
	Observe that $\overline{\mathcal{S}}_{k}\simeq\bigoplus_{mN=0}^{\infty}\overline{\mathcal{S}}_{k}^{(N)}$ as $(1/m)\mathbb{Z}_{\geq 0}$\text{-}graded algebra over $\mathbb{F}_{q}$.  Assuming $d_{k}(n,m,p)=0$, by Theorem \ref{d_{k}(n,m,p)=0} we have the following exact sequence:
	$$0\longrightarrow\overline{\mathcal{S}}_{k}\xlongrightarrow[]{\phi_{\overline{G}}}\overline{\mathcal{S}}_{k}\frac{dt}{t}\longrightarrow\overline{\mathcal{S}}_{k}\frac{dt}{t}/\phi_{\overline{G}}\overline{\mathcal{S}}_{k}=H^{1}(\overline{\mathcal{S}}_{k},\phi_{\overline{G}})\longrightarrow0.$$
	Let $P(T)$ be the Poincar\'e series of $\overline{\mathcal{S}}_{k}$, and $Q(T)$ the Poincar\'e series of $H^{1}(\overline{\mathcal{S}}_{k},\phi_{\overline{G}})$. The Poincar\'e series of Image($\phi_{\overline{G}}$) is $T^{m}P(T)$ since $\phi_{\overline{G}}$ is injective and of weight $1$. Hence $Q(T)=(1-T^{m})P(T)$.
	
	For each $N\in(1/m)\mathbb{Z}_{\geq 0}$, we let $d_{N}:=\#\{\underline{i}\in I_{k}|\omega(\underline{i})=N\}$. Then $P(T)$ may be  written as
	\begin{equation*}
		\begin{aligned}
			P(T)&=\sum_{N\in(1/m)\mathbb{Z}_{\geq 0}}\dim_{\mathbb{F}_{q}}\overline{\mathcal{S}}_{k}^{(N)}T^{mN}
			=\sum_{N\in(1/m)\mathbb{Z}_{\geq 0}}(\sum_{r=0}^{\infty}d_{N-\frac{(mn+1)r}{m}})T^{mN}\\
			&=\sum_{N\in(1/m)\mathbb{Z}_{\geq 0}}\sum_{r=0}^{\infty}d_{M}T^{mM+(mn+1)r}
			=(\sum_{N\in(1/m)\mathbb{Z}_{\geq 0}}d_{M}T^{mM})(\sum_{r=0}^{\infty}T^{(mn+1)r}) \\
			&=\frac{\sum_{\underline{i}\in I_{k}}T^{m\omega(\underline{i})}}{1-T^{mn+1}}.
		\end{aligned}
	\end{equation*}
	Let $R(T):=\sum_{\underline{i}\in I_{k}}T^{\sum_{j=0}^{mn}j\cdot i_{j}}$.
	Then $P(T)=\frac{R(T)}{1-T^{mn+1}}$ and $Q(T)=\frac{(1-T^{m})R(T)}{1-T^{mn+1}}$.
	
	\begin{thm} \label{Theorem 3.3}
		Suppose $d_{k}(n,m,p)=0$, then $H^{1}(\overline{\mathcal{S}}_{k},\phi_{\overline{G}})$ is an $\mathbb{F}_{q}$\text{-}vector space of dimension $\frac{m}{mn+1}\binom{mn+k}{k}$.
	\end{thm}
	\begin{proof}
		When $d_k (n,m,p)=0$, we claim that $\gcd(mn+1,k)=1$, otherwise let $d=\gcd(mn+1,k)$, let
		\begin{equation*}
			i_j =\begin{cases}
				\frac{k}{d} & j=\frac{\alpha(mn+1)}{d}\ \mathrm{for}\ 0\leq \alpha\leq d-1, \\
				0 & \mathrm{otherwise}.
			\end{cases}    
		\end{equation*}
		When $p\nmid mn+1$, we have $p\nmid d$, so the primitive $d$-th root of unity,$\bar{\zeta}_d$ ,exists in $\overline{\mathbb{F}}_p$. One immediately obtains that $(i_0 ,\cdots, i_{mn})\in I_k$ and 
		$$\sum_{j=0}^{mn}i_{j}\bar{\zeta}_{mn+1}^{j}=\frac{k}{d}\sum_{\alpha=0}^{d-1}\bar{\zeta}_{d}^{\alpha}=0$$
		in $\overline{\mathbb{F}}_p$, contradicting $d_{k}(n,m,p)=0$, so we have $\gcd(mn+1,k)=1$. Apply [\cite{HS24}, Lemma 4.8], $$Q(T)=\frac{(1-T^m)R(T)}{1-T^{mn+1}}=\frac{(1+T+\cdots +T^{m-1})R(T)}{1+T+\cdots +T^{mn}}$$ is a polynomial with integer coefficients.
		Note that $R(1)=\binom{mn+k}{k}$ is the $\mathbb{F}_{q}[t]$\text{-}rank of $\overline{\mathcal{S}}_{k}$. Then
		$$\dim_{\mathbb{F}_{q}}H^{1}(\overline{\mathcal{S}}_{k},\phi_{\overline{G}})=Q(1)=\frac{m}{mn+1}R(1)=\frac{m}{mn+1}\binom{mn+k}{k}.$$
	\end{proof}
	
	\begin{rmk} \label{rmk4.5}
		Under the assumption $d_k (n,m,p)=0$, follows from the calculation of the Hilbert series $Q(T)$ we immediately have 
		$$h_{i}(n,m,k)=\#\{(r,\underline{j})\in B_k |W(r;\underline{j})=\frac{i}{m}\}$$
		are non-negative integers and $h_i (n,m,k)$ are independent of the choice of basis index set $B_k$. This is why in the introduction we call the lower convex hull of $(0,0)$ and points in (\ref{Hodge}) combinatorial Hodge polygon. We hope further research in irregular Hodge theory will give a geometric interpretation for the numbers $h_i (n,m,k)$ as in the $Kl_{n,1}$ case studied by \cite{Qin24}.
	\end{rmk}
	
	Now Let $B_{k}$ be a subset of $\mathbb{Z}_{\geq 0}\times I_{k}$ with cardinality $\frac{m}{mn+1}\binom{mn+k}{k}$ such that $\{t^{r}\bar{e}^{\underline{i}}\}_{(r,\underline{i})\in B_{k}}$ is a basis of $H^{1}(\overline{\mathcal{S}}_{k},\phi_{\overline{G}})$. Note that 
	$\phi_{\overline{G}}$ maps weight $N$ elements to weight $N+1$. We then have the following result.
	\begin{cor} \label{Corollary 3.4}
		Let $d_{k}(n,m,p)=0$. For $\bar{\xi}\in\overline{\mathcal{S}}_{k}^{(N)}$, there exist $\bar{a}(r,\underline{i})\in\mathbb{F}_{q}$ with $(r,\underline{i})\in B_{k}$, $W(r;\underline{i})=N$, and $\bar{\zeta}\in\overline{\mathcal{S}}_{k}^{(N-1)}$, such that
		$$\bar{\xi}=\sum_{\begin{subarray}{c}
				(r,\underline{i})\in B_{k}  \\
				W(r;\underline{i})=N
		\end{subarray}}\bar{a}(r,\underline{i})t^{r}\bar{e}^{\underline{i}}+\phi_{\overline{G}}(\bar{\zeta}).$$
	\end{cor}
	\par Recall $\overline{\nabla}^{(1)}=t\frac{d}{dt}+\phi_{\overline{G}}$. Combining Theorem \ref{d_{k}(n,m,p)=0}, Theorem \ref{Theorem 3.3} and Corollary \ref{Corollary 3.4}, we obtain:
	\begin{thm} \label{Theorem 3.5}
		Let $d_{k}(n,m,p)=0$. Then $H^{0}(\overline{\mathcal{S}}_{k},\overline{\nabla}^{(1)})=0$  and $H^{1}(\overline{\mathcal{S}}_{k},\overline{\nabla}^{(1)})$ is an $\mathbb{F}_{q}$\text{-}vector space of dimension $\frac{m}{mn+1}\binom{mn+k}{k}$ with basis $\{t^{r}\bar{e}^{\underline{i}}\}_{(r,\underline{i})\in B_{k}}$. For any $\bar{\xi}\in\Fil^{N}\overline{\mathcal{S}}_{k}$, there exist $\bar{a}(r,\underline{i})\in\mathbb{F}_{q}$ with $(r,\underline{i})\in B_{k}$, $W(r;\underline{i})\le N$ with $\bar{\zeta}\in\Fil^{(N-1)}\overline{\mathcal{S}}_{k}$, such that
		$$\bar{\xi}=\sum_{\begin{subarray}{c}
				(r,\underline{i})\in B_{k}  \\
				W(r;\underline{i})\leq N
		\end{subarray}}\bar{a}(r,\underline{i})t^{r}\bar{e}^{\underline{i}}+\overline{\nabla}^{(1)}(\bar{\zeta}).$$
	\end{thm}
	
	\subsection{Symmetric power cohomology} 
	
	Let $\mathcal{S}_{k}$ denote $\mathcal{S}_{k,t^{p^{s}}}$ for the case $s=0$.
	Then from the definition of $\mathcal{S}_{k}$ we may write it as 
	$$\mathcal{S}_{k}=\{\sum_{(r,\underline{i})\in\mathbb{Z}_{\geq 0}\times I_{k}}a(r,\underline{i})\gamma^{\frac{(mn+1)r}{m}}t^{r}e^{\underline{i}}|a(r,\underline{i})\in\mathfrak{D}_q, a(r,\underline{i})\rightarrow 0\ {\rm as}\ W(r,\underline{i})\rightarrow \infty\}.$$
	For each $N\in(1/m)\mathbb{Z}_{\geq 0}$, define an increasing filtration on $\mathcal{S}_{k}$ by
	$$\Fil^{N}\mathcal{S}_{k}=\{\mathfrak{D}_q\text{-}\mathrm{module}\ \mathrm{generated}\ \mathrm{by}\ \gamma^{\frac{(mn+1)r}{m}}t^{r}e^{\underline{i}}\ \mathrm{with}\ (r,\underline{i})\in\mathbb{Z}_{\geq 0}\times I_{k}\  {\rm and}\ W(r;\underline{i})\leq N\}$$
	We linearly extend $\nabla^{(1)}$ to $\mathcal{S}_{k}$ by
	$$\nabla^{(1)}(\gamma^{\frac{(mn+1)r}{m}}t^{r}e^{\underline{i}}):=r\gamma^{\frac{(mn+1)r}{m}}t^{r}e^{\underline{i}}+\gamma^{\frac{(mn+1)r}{m}}t^{r}\sum_{l=0}^{mn}i_{l}e_{0}^{i_{0}}\cdot\cdot\cdot e_{l}^{i_{l}-1}\cdot\cdot\cdot e_{mn}^{i_{mn}}\nabla^{(1)}(e_{l}).$$
	Note that the complex $\Omega^{\bullet}(\mathcal{S}_{k},\nabla)\ \mathrm{mod}\ \gamma^{1/mq}$ with respect to the basis $\{\gamma^{\frac{(mn+1)r}{m}}t^{r}e^{\underline{i}}\}_{(r,\underline{i})\in\mathbb{Z}_{\geq 0}\times I_{k}}$ is isomorphic to the complex $\Omega^{\bullet}(\overline{\mathcal{S}}_{k},\overline{\nabla}^{(1)})$ under the reduction $a(r,\underline{i})\gamma^{\frac{(mn+1)r}{m}}t^{r}e^{\underline{i}}\mapsto\bar{a}(r,\underline{i})t^{r}\bar{e}^{\underline{i}}$.
	Then a standard lifting argument together with Theorem \ref{Theorem 3.5} show the following.
	\begin{lem} \label{Lemma 3.6}
		Let $d_{k}(n,m,p)=0$. Then for any $\xi\in\Fil^{N}\mathcal{S}_{k}$, there exist $a(r,\underline{i})\in\mathfrak{D}_q$ with $(r,\underline{i})\in B_{k}$ and $W(r;\underline{i})\leq N$, and $\zeta\in\Fil^{N-1}\mathcal{S}_{k}$, such that
		$$\xi=\sum_{\begin{subarray}{c}
				(r,\underline{i})\in B_{k}  \\
				W(r;\underline{i})\leq N
		\end{subarray}}a(r,\underline{i})\gamma^{\frac{(mn+1)r}{m}}t^{r}e^{\underline{i}}+\nabla^{(1)}(\zeta).$$
	\end{lem}
	Passing from $\nabla^{(1)}$ to $\nabla$, we obtain:
	\begin{lem} \label{Lemma 3.7}
		Let $d_{k}(n,m,p)=0$. Then for any $\xi\in\Fil^{N}\mathcal{S}_{k}$, there exist $a(r,\underline{i})\in\mathfrak{D}_q$, $\zeta\in\Fil^{N-1}\mathcal{S}_{k}$, and $\varpi(\textbf{r})\in\Fil^{N+\rho(\textbf{r})}\mathcal{S}_{k}$ for $\textbf{r}\in\mathcal{F}_{i}$ such that
		$$\xi=\sum_{\begin{subarray}{c}
				(r,\underline{i})\in B_{k}  \\
				W(r;\underline{i})\leq N
		\end{subarray}}a(r,\underline{i})\gamma^{\frac{(mn+1)r}{m}}t^{r}e^{\underline{i}}+\nabla(\zeta)+\sum_{i=1}^{\infty}\sum_{\textbf{r}\in\mathcal{F}_{i}}p^{\rho(\textbf{r})}\varpi(\textbf{r}).$$
	\end{lem}
	\begin{proof}
		By Lemma \ref{Lemma 3.6}, there exist $a(r,\underline{i})\in\mathfrak{D}_q$ with $(r,\underline{i})\in B_{k}$ and $W(r,\underline{i})\leq N$, and $\zeta\in\Fil^{N-1}\mathcal{S}_{k}$, such that
		$$\xi=\sum_{\begin{subarray}{c}
				(r,\underline{i})\in B_{k}  \\
				W(r;\underline{i})\leq N
		\end{subarray}}a(r,\underline{i})\gamma^{\frac{(mn+1)r}{m}}t^{r}e^{\underline{i}}+\nabla^{(1)}(\zeta).$$
		Now we write
		$$\zeta=\sum_{\begin{subarray}{c}
				(r,\underline{i})\in \mathbb{Z}_{\geq 0}\times I_{k}  \\
				W(r;\underline{i})\leq N-1
		\end{subarray}}b(r,\underline{i})\gamma^{\frac{(mn+1)r}{m}}t^{r}e^{\underline{i}}$$
		with $b(r,\underline{i})\in\mathfrak{D}_q$.
		Applying Lemma \ref{Lemma 2.11}, we obtain that
		\begin{equation*}
			\begin{aligned}
				\nabla^{(1)}(\zeta)&=t\frac{d \zeta}{dt}+\sum_{\begin{subarray}{c}
						(r,\underline{i})\in \mathbb{Z}_{\geq 0}\times I_{k}  \\
						W(r;\underline{i})\leq N-1
				\end{subarray}}b(r,\underline{i})\gamma^{\frac{(mn+1)r}{m}}t^{r}\sum_{l=0}^{mn}i_{l}e_{0}^{i_{0}}\cdot\cdot\cdot e_{l}^{i_{l}-1}\cdot\cdot\cdot e_{mn}^{i_{mn}}\nabla^{(1)}(e_{l})   \\
				&=t\frac{d\zeta}{dt}+\sum_{\begin{subarray}{c}
						(r,\underline{i})\in\mathbb{Z}_{\geq 0}\times I_{k}  \\
						W(r;\underline{i})\leq N-1
				\end{subarray}}b(r,\underline{i})\gamma^{\frac{(mn+1)r}{m}}t^{r}\sum_{l=0}^{mn}i_{l}e_{0}^{i_{0}}\cdot\cdot\cdot e_{l}^{i_{l}-1}\cdot\cdot\cdot e_{mn}^{i_{mn}}\Big(\nabla(e_{l}) \\
				&+\sum_{i=1}^{\infty}\sum_{\textbf{r}\in\mathcal{F}_{i}}\sum_{j=0}^{\min\{mn,m\rho(\textbf{r})+l+m\}}p^{\rho(\textbf{r})}b(j,l;\textbf{r})e_{j}\Big) \\
				&=\nabla(\zeta)+\sum_{i=1}^{\infty}\sum_{\textbf{r}\in\mathcal{F}_{i}}p^{\rho(\textbf{r})}\varpi(\textbf{r})
			\end{aligned}
		\end{equation*}
		where
		$$\varpi(\textbf{r})=\sum_{l=0}^{mn}\sum_{j=0}^{\min\{mn,m\rho(\textbf{r})+l+m\}}b(j,l;\textbf{r})\sum_{\begin{subarray}{c}
				(r,\underline{i})\in \mathbb{Z}_{\geq 0}\times I_{k}  \\
				W(r;\underline{i})\leq N-1
		\end{subarray}}b(r,\underline{i})\gamma^{\frac{(mn+1)r}{m}}t^{r}i_{l}e_{0}^{i_{0}}\cdot\cdot\cdot e_{l}^{i_{l}-1}\cdot\cdot\cdot e_{mn}^{i_{mn}}e_{j}$$
		with $b(j,l;\textbf{r})\in\Fil^{\rho(\textbf{r})+1+(l-j)/m}\mathcal{O}$.
		Then the weight of each summand of $\varpi(\textbf{r})$ admits an upper bound
		$$\rho(\textbf{r})+1+\frac{l-j}{m}+\frac{(mn+1)r}{m}+\omega(\underline{i})+\frac{j}{m}-\frac{l}{m}=W(r,\underline{i})+1+\rho(\textbf{r})\leq N+\rho(\textbf{r}).$$
		Hence $\varpi(\textbf{r})\in {\rm Fil}^{N+\rho(\mathbf{r})}$.
		This finishes the proof of Lemma \ref{Lemma 3.7}.

	\end{proof}
	
	\begin{thm} \label{Theorem 3.8}
		Suppose $d_{k}(n,m,p)=0$, then $H^{0}(\mathcal{S}_{k},\nabla)=0$ and $H^{1}(\mathcal{S}_{k},\nabla)$ is a free $\mathfrak{D}_q$\text{-}module of rank $\frac{m}{mn+1}\binom{mn+k}{k}$ with basis $\{\gamma^{\frac{(mn+1)r}{m}}t^{r}e^{\underline{i}}\}_{(r,\underline{i})\in B_{k}}$. Furthermore, for any $\xi\in\Fil^{N}\mathcal{S}_{k}$, there exists $C(r,\underline{i};\xi)\in\mathfrak{D}_q$ with $\ord_{p}C(r,\underline{i};\xi)\geq W(r;\underline{i})-N$, such that in $H^{1}(\mathcal{S}_{k},\nabla)$
		$$\xi=\sum_{(r,\underline{i})\in B_{k}}C(r,\underline{i};\xi)\gamma^{\frac{(mn+1)r}{m}}t^{r}e^{\underline{i}}.$$
	\end{thm}
	\begin{proof}
		It follows from Theorem \ref{Theorem 3.5} together with a standard lifting that $H^{0}(\mathcal{S}_{k},\nabla)=0$ and $H^{1}(\mathcal{S}_{k},\nabla)$ is a free $\mathfrak{D}_q$\text{-}module with basis $\{\gamma^{\frac{(mn+1)r}{m}}t^{r}e^{\underline{i}}\}_{(r,\underline{i})\in B_{k}}$.
		
		Let $\xi\in\Fil^{N}\mathcal{S}_{k}$.
		By Lemma \ref{Lemma 3.7}, we have
		$$\xi=\sum_{\begin{subarray}{c}
				(r,\underline{i})\in B_{k}  \\
				W(r;\underline{i})\leq N
		\end{subarray}}a(r,\underline{i})\gamma^{\frac{(mn+1)r}{m}}t^{r}e^{\underline{i}}+\sum_{i=1}^{\infty}\sum_{\textbf{r}_{1}\in\mathcal{F}_{i}}p^{\rho(\textbf{r}_{1})}\varpi(\textbf{r}_{1})$$
		in $H^{1}(\mathcal{S}_{k},\nabla)$, where $a(r,\underline{i})\in\mathfrak{D}_q$ and $\omega(\textbf{r}_{1})\in\Fil^{N+\rho(\textbf{r}_{1})}\mathcal{S}_{k}$. Applying Lemma \ref{Lemma 3.7} to $\varpi(\textbf{r}_{1})$ we have
		$$\varpi(\textbf{r}_{1})=\sum_{\begin{subarray}{c}
				(r,\underline{i})\in B_{k}  \\
				W(r;\underline{i})\leq N+\rho(\textbf{r}_{1})
		\end{subarray}}a^{(\textbf{r}_{1})}(r,\underline{i})\gamma^{\frac{(mn+1)r}{m}}t^{r}e^{\underline{i}}+\sum_{i=1}^{\infty}\sum_{\textbf{r}_{2}\in\mathcal{F}_{i}}p^{\rho(\textbf{r}_{2})}\varpi(\textbf{r}_{1},\textbf{r}_{2})$$
		with $a^{(\textbf{r}_{1})}(r,\underline{i})\in\mathfrak{D}_q$ and $\varpi(\textbf{r}_{1},\textbf{r}_{2})\in\Fil^{N+\rho(\textbf{r}_{1})+\rho(\textbf{r}_{2})}\mathcal{S}_{k}$. Recursively applying Lemma \ref{Lemma 3.7} we get
		$$\varpi(\textbf{r}_{1},\cdot\cdot\cdot,\textbf{r}_{l})=\sum_{\begin{subarray}{c}
				(r,\underline{i})\in B_{k}  \\
				W(r;\underline{i})\leq N+\rho(\textbf{r}_{1})+\cdot\cdot\cdot+\rho(\textbf{r}_{l})
		\end{subarray}}a^{(\textbf{r}_{1},\cdot\cdot\cdot,\textbf{r}_{l})}(r,\underline{i})\gamma^{\frac{(mn+1)r}{m}}t^{r}e^{\underline{i}}+\sum_{i=1}^{\infty}\sum_{\textbf{r}_{l+1}\in\mathcal{F}_{i}}p^{\rho(\textbf{r}_{l+1})}\varpi(\textbf{r}_{1},\cdot\cdot\cdot,\textbf{r}_{l+1})$$
		with $a^{(\textbf{r}_{1},\cdot\cdot\cdot,\textbf{r}_{l})}(r,\underline{i})\in\mathfrak{D}_q$ and $\varpi(\textbf{r}_{1},\cdot\cdot\cdot,\textbf{r}_{l+1})\in\Fil^{N+\rho(\textbf{r}_{1})+\cdot\cdot\cdot+\rho(\textbf{r}_{l+1})}\mathcal{S}_{k}$.
		Then we have
		$$\xi=\sum_{(r,\underline{i})\in B_{k}}C^{(l)}(r,\underline{i};\xi)\gamma^{\frac{(mn+1)r}{m}}t^{r}e^{\underline{i}}+\sum_{i_{1},\cdot\cdot\cdot,i_{l+1}\geq 1}\sum_{\begin{subarray}{c}
				\textbf{r}_{j}\in\mathcal{F}_{i_{j}}\\
				1\leq j\leq l+1
		\end{subarray}}p^{\rho(\textbf{r}_{1})+\cdot\cdot\cdot\rho(\textbf{r}_{l+1})}\varpi(\textbf{r}_{1},\cdot\cdot\cdot,\textbf{r}_{l+1})$$
		with
		$$C^{(l)}(r,\underline{i};\xi)=\widetilde{a}(r,\underline{i})+\sum_{j=1}^{l}\sum_{i_{1},\cdot\cdot\cdot,i_{j}\geq 1}\sum_{\begin{subarray}{c}
				\textbf{r}_{\alpha}\in\mathcal{F}_{i_{\alpha}}\\
				1\leq\alpha\leq j
		\end{subarray}}p^{\rho(\textbf{r}_{1})+\cdot\cdot\cdot\rho(\textbf{r}_{j})}\widetilde{a}^{(\textbf{r}_{1},\cdot\cdot\cdot,\textbf{r}_{j})}(r,\underline{i}),$$
		where 
		\begin{equation*}
			\widetilde{a}(r,\underline{i})=\begin{cases}
				a(r,\underline{i}) & \mathrm{when}\ W(r;\underline{i})\leq N ,   \\
				0 & \mathrm{otherwise},
			\end{cases}    
		\end{equation*} and
		\begin{equation*}
			\widetilde{a}^{(\textbf{r}_{1},\cdot\cdot\cdot,\textbf{r}_{j})}(r,\underline{i})=\begin{cases}
				a^{(\textbf{r}_{1},\cdot\cdot\cdot,\textbf{r}_{j})}(r,\underline{i}) & \mathrm{when}\  W(r;\underline{i})\leq N+\rho(\textbf{r}_{1})+\cdot\cdot\cdot+\rho(\textbf{r}_{j}), \\
				0 & \mathrm{otherwise}.
			\end{cases}    
		\end{equation*}
		We see that
		${\rm ord}_{p}\widetilde{a}(r,\underline{i})\geq 0\geq W(r;\underline{i})-N$, and
		$${\rm ord}_{p}p^{\rho(\textbf{r}_{1})+\cdots\rho(\textbf{r}_{j})}\widetilde{a}^{(\textbf{r}_{1},\cdots,\textbf{r}_{j})}(r,\underline{i})\geq\rho(\textbf{r}_{1})+\cdot\cdot\cdot\rho(\textbf{r}_{j})\geq W(r;\underline{i})-N.$$
		This shows $\ord_{p}C^{(l)}(r,\underline{i};\xi)\geq W(r;\underline{i})-N$ for every $l\geq 1$. 
		
		Note that for any $\mathbf{r}=(r_1 ,\cdots, r_i , 0 , \cdots)\in\mathcal{F}_i$, $\rho(\mathbf{r})=\sum_{s=1}^{i}(p^{r_s}-1)\geq i(p-1)$, thus the terms in $C^{(l+1)}(r,\underline{i};\xi)-C^{(l)}(r,\underline{i};\xi)$ admit $p$-adic valuations at least $(l+1)(p-1)$, the series $\{C^{(l)}(r,\underline{i};\xi)\}_{l\geq 1}$ convergence $p$-adically in $\mathfrak{D}_q$. For the same reason, those error terms in the expression of $\xi$, $p^{\rho(\mathbf{r}_1)+\cdots+\rho(\mathbf{r}_{l+1})}\varpi(\mathbf{r}_1 ,\cdots, \mathbf{r}_{l+1})$, will $p$-adically tends to $0$ as $l\rightarrow+\infty$.
		
		The thereom then follows by setting $l\rightarrow+\infty$ and $C(r,\underline{i};\xi):=\lim_{l\rightarrow\infty}C^{(l)}(r,\underline{i};\xi)$.
	\end{proof}

	\section{Frobenius estimates}

	Recall that $\bar{\alpha}_{1}(t)=\Frob^{n}(\alpha_{1}(t)):H^{n}(\mathcal{C},D_{t})\rightarrow H^{n}(\mathcal{C}_{p},D_{t^{p}})$. For each $0\le i\le mn$, write $$\bar{\alpha}_{1}(t)(e_{i}):=\sum_{j=0}^{mn}\widetilde{A}(j,i)e_{j}$$
	with coefficient $\widetilde{A}(j,i)\in\mathcal{O}_{p}$. Then \cite[Corollary 3.6]{HS172} gives the following $p$-adic estimates.
	\begin{lem} \label{Lemma 2.10}
		For $0\leq i,j\leq mn$, we have $\widetilde{A}(j,i)\in\mathcal{O}_{p}(\frac{mn+1}{mp};\frac{j}{m})$.
	\end{lem}
	For each $\underline{i}\in I_k$, we write
	$$\Sym^{k}\bar{\alpha}_{1}(t)(e^{\underline{i}})=\sum_{\underline{j}\in I_{k}}\widetilde{A}^{(k)}(\underline{j},\underline{i})e^{\underline{j}}$$
	with the coefficients $\widetilde{A}^{(k)}(\underline{j},\underline{i})\in\mathcal{O}_{p}$. 
	We obtain the $p$-adic estimate for the symmetric power Frobenius from \cite[Proposition 4.1]{HS172}.
	\begin{lem} \label{Lemma 3.9}
		For $\underline{j},\underline{i}\in I_{k}$, we have $\widetilde{A}^{(k)}(\underline{j},\underline{i})\in\mathcal{O}_{p}(\frac{mn+1}{mp};\omega(\underline{j}))$.
	\end{lem}

			We use the $\mathfrak{D}_q$\text{-}basis $\{\gamma^{\frac{(mn+1)r}{m}}t^{r}e^{\underline{i}}\}_{(r,\underline{i})\in B_{k}}$ of $H^{1}(\mathcal{S}_{k},\nabla)$ to estimate the Frobenius action. Recall that $\beta_{k,1}=\psi_{t}\circ\Sym^{k}\bar{\alpha}_{1}(t):\mathcal{S}_{k}\rightarrow\mathcal{S}_{k}$.
			Since $\beta_{k,1}\circ\nabla=p\nabla\circ\beta_{k,1}$, it induces $\bar{\beta}_{k,1}:H^{1}(\mathcal{S}_{k},\nabla)\rightarrow H^{1}(\mathcal{S}_{k},\nabla)$. 
			For $(s,\underline{j}),(r,\underline{i})\in B_k$, we write
			$$\bar{\beta}_{k,1}(\gamma^{\frac{(mn+1)r}{m}}t^{r}e^{\underline{i}}):=\sum_{(s,\underline{j})\in B_{k}}B^{(k)}\big((s,\underline{j}),(r,\underline{i})\big)\gamma^{\frac{(mn+1)s}{m}}t^{s}e^{\underline{j}}$$
			with the coefficient $B^{(k)}\big((s,\underline{j}),(r,\underline{i})\big)\in\mathfrak{D}_q$.
			
			\begin{thm} \label{Theorem 3.10}
				Let $d_{k}(n,m,p)=0$. Then for any $(s,\underline{j}),(r,\underline{i})\in B_{k}$, ${\rm ord}_{p}B^{(k)}\big((s,\underline{j}),(r,\underline{i})\big)\geq W(s;\underline{j})$.
			\end{thm}
			
			\begin{proof}
				It follows from the definition and Lemma \ref{Lemma 3.9} that
				\begingroup
				\allowdisplaybreaks
				\begin{align*}
					\beta_{k,1}(\gamma^{\frac{(mn+1)r}{m}}t^{r}e^{\underline{i}})& =\psi_{t}(\sum_{\underline{i}'\in I_{k}}\gamma^{\frac{(mn+1)r}{m}}t^{r}\widetilde{A}^{(k)}(\underline{i}',\underline{i})e^{\underline{i}'}) \\
					&=\psi_{t}(\sum_{\underline{i}'\in I_{k},\ l\geq 0}\gamma^{\frac{(mn+1)(p-1)r}{mp}}\widetilde{A}^{(k)}(\underline{i}',\underline{i};l)\gamma^{\frac{(mn+1)(l+r)}{mp}}t^{l+r}e^{\underline{i}'}) \\
					&=\psi_{t}(\sum_{\underline{i}'\in I_{k},\ l\geq r}\gamma^{\frac{(mn+1)(p-1)r}{mp}}\widetilde{A}^{(k)}(\underline{i}',\underline{i};l-r)\gamma^{\frac{(mn+1)l}{mp}}t^{l}e^{\underline{i}'}) \\
					&=\sum_{\underline{i}'\in I_{k},\ l\geq \lceil\frac{r}{p}\rceil}\gamma^{\frac{(mn+1)(p-1)r}{mp}}\widetilde{A}^{(k)}(\underline{i}',\underline{i};pl-r)\gamma^{\frac{(mn+1)l}{m}}t^{l}e^{\underline{i}'},
				\end{align*}
				\endgroup
				where $\widetilde{A}^{(k)}(\underline{i}',\underline{i})=\sum_{l\geq 0}\widetilde{A}^{(k)}(\underline{i}',\underline{i};l)\gamma^{\frac{(mn+1)l}{mp}}t^{l}$. Lemma \ref{Lemma 3.9} shows that $\widetilde{A}^{(k)}(\underline{i}',\underline{i})$ is a well-defined element in $\mathcal{O}_{p}(\frac{mn+1}{mp};\omega(\underline{i}'))$ with $\ord_{p}\widetilde{A}^{(k)}(\underline{i}',\underline{i};l)\geq\frac{(mn+1)l}{mp}+\omega(\underline{i}')$. Note that $\gamma^{\frac{(mn+1)l}{m}}t^{l}e^{\underline{i}'}\in\Fil^{W(l;\underline{i}')}\mathcal{S}_{k}$. By Theorem \ref{Theorem 3.8} we have
				$$\gamma^{\frac{(mn+1)l}{m}}t^{l}e^{\underline{i}'}=\sum_{(s,\underline{j})\in B_{k}}C(s,\underline{j};l,\underline{i}')\gamma^{\frac{(mn+1)s}{m}}t^{s}e^{\underline{j}}\ \ \mathrm{mod}\ \nabla$$
				with ${\rm ord}_{p}C(s,\underline{j};l,\underline{i}')\geq W(s;\underline{j})-W(l;\underline{i}')$.
				Then
				$$B^{(k)}\big((s,\underline{j}),(r,\underline{i})\big)=\sum_{\underline{i}'\in I_{k},\ l\geq \lceil\frac{r}{p}\rceil}\gamma^{\frac{(mn+1)(p-1)r}{mp}}\widetilde{A}^{(k)}(\underline{i}',\underline{i};pl-r)C(s,\underline{j};l,\underline{i}').$$
				Theorem \ref {Theorem 3.10} holds from
				\begingroup
				\allowdisplaybreaks
				\begin{align*}
					{\rm ord}_{p}&\gamma^{\frac{(mn+1)(p-1)r}{mp}}\widetilde{A}^{(k)}(\underline{i}',\underline{i};pl-r)C(s,\underline{j};l,\underline{i}') \\ &\geq\frac{(mn+1)r}{mp}+\frac{(mn+1)(pl-r)}{mp}+\omega(\underline{i}')+W(s;\underline{j})-W(l;\underline{i}')=W(s;\underline{j}).
				\end{align*}
				\endgroup
				
			\end{proof}
			
			\begin{thm} \label{thm0110}
				Suppose $p\nmid m(mn+1)$ and $d_{k}(n,m,p)=0$. Then $L({\rm Sym}^{k}Kl_{n,m}/\mathbb{F}_{q},T)$ is a polynomial in $1+T\mathbb{Z}[\zeta_{p}][T]$ of degree at most $\frac{m}{mn+1}\binom{mn+k}{mn}$, and its $q$-adic Newton polygon lies on or above the $q$-adic Newton polygon of $\prod(1-q^{i/m}T)^{h_{i}(n,m,k)}$.
			\end{thm}
			\begin{proof}
				By Lemma \ref{Lemma 3.1} we know $\bar{\beta}_{k,1}^{a}=\bar{\beta}_{k,a}$. Since $d_{k}(n,m,p)=0$, (\ref{(3.2)}) and Theorem \ref{Theorem 3.8} show that $$L({\rm Sym}^{k}Kl_{n,m}/\mathbb{F}_{q},T)=\det(1-\bar{\beta}_{k,a}T|H^{1}(\mathcal{S}_{k},\nabla)).$$
				It follows from Theorem \ref{Theorem 3.8} that  $L({\rm Sym}^{k}Kl_{n,m}/\mathbb{F}_{q},T)$ is a polynomial of degree at most $\frac{m}{mn+1}\binom{mn+k}{mn}$.

				On the other hand, we have that 
				\begin{align}\label{eq44}
					\det(1-\bar{\beta}_{k,a}T^a|H^{1}(\mathcal{S}_k,\nabla))&=\det (1-\bar{\beta}_{k,1}^aT^a|H^{1}(\mathcal{S}_k,\nabla))\\\nonumber
					&=\prod_{\zeta^a=1} \det (1-\zeta\bar{\beta}_{k,1}T|H^{1}(\mathcal{S}_k,\nabla)).
				\end{align}
				Let $m_i$ denote the number of reciprocal roots of $\det (1-\bar{\beta}_{k,1}T|H^{1}(\mathcal{S}_k,\nabla))$ which has slope $s_i$. Then by (\ref{eq44}) we have that $\det(1-\bar{\beta}_{k,a}T|H^{1}(\mathcal{S}_k,\nabla))$ has $m_i$ reciprocal roots of $q$-adic slope $s_i$.
				Hence Theorem \ref{thm0110}  follows from Theorem \ref{Theorem 3.10} and Dwork's classical argument \cite[Section 7]{Dw64}.
			\end{proof}

			\begin{rmk}
				Our main result asserts only the degree of $L({\rm Sym}^{k}Kl_{n,m}/\mathbb{F}_{q},T)$ is at most $\frac{m}{mn+1}\binom{nm+k}{k}$, since we are unable to prove $\bar{\beta}_{k,a}$ has full rank on $H^1 (\mathcal{S}_k , \nabla)$. We also do not know whether it is always a polynomial without the condition $d_k (n,m,p)=0$. However, as in the case $Kl_{n,1}$ studied in \cite{FW1}, it is reasonable to predict the degree of the rational function $L({\rm Sym}^{k}Kl_{n,m}/\mathbb{F}_{q},T)$(degree of the numerator subtract the degree of the denominator) is 
				$$\frac{m}{mn+1}\left(\binom{mn+k}{k}-d_{k}(n,m,p)\right)$$
				given that $p\nmid m(mn+1)$. This will confirm the equality of the degree in our main result. Further problems include the determination of trivial factors and functional equations, as well as a potential uniform lower bounds of the Newton polygon without the condition $d_k(n,m,p)=0$.
			\end{rmk}
			
			\bibliographystyle{amsplain}

\begin{thebibliography}{10}
				\bibitem{AS1} A. Adolphson and S. Sperber, Exponential sums and Newton
				polyhedra: cohomology and estimates, {\it Ann. of Math.} {\bf 130} (1989), 367-406.
				\bibitem{Dw64} B. Dwork, On the zeta function of a hypersurface.II, {\it Ann. of Math.} {\bf 80} (1964),227-299.
				
				\bibitem{FSY}J. Fres\'{a}n, C. Sabbah, and J.-D. Yu, Hodge theory of Kloosterman connections, {\it Duke Math. J.} {\bf171} (2022), 1649-1747.
				
				
				\bibitem{FW1} L. Fu and D. Wan, $L$-functions for symmetric products of Kloosterman sums, {\it J. Reine Angew. Math.}    {\bf589} (2005),
				79-103.
				
				\bibitem{FW2} L. Fu and D. Wan, $L$-functions of symmetric products of the Kloosterman sheaf over $\mathbf{Z}$, {\it Math. Ann.} {\bf 342} (2008), 387-404.
				
				\bibitem{FW3} L. Fu and D. Wan, Trivial factors for $L$-functions of symmetric products of Kloosterman sheaves, {\it Finite Fields Appl.} {\bf 14} (2008), 549-570.
				
				\bibitem{FW23} L. Fu and P. Li, D. Wan, H. Zhang, $p$-adic GKZ hypergeometric complex, Math. Ann. {\bf 387} (2023), no.~3-4, 1629--1689; MR4657433
				
				\bibitem{H14} C. D. Haessig, Meromorphy of the rank one unit root $L$-function revisited, {\it Finite Fields Appl.} {\bf 30} (2014), 191-202.
				
				\bibitem{H17} C. D. Haessig,  $L$-functions of symmetric powers of Kloosterman sums (unit root $L$-functions and $p$-adic estimates), {\it Math. Ann.} {\bf 369} (2017), 17-47.
				
				\bibitem{HS172} C. D. Haessig and S. Sperber, Symmetric power $L$-functions for families of generalized Kloosterman sums, {\it Trans. Amer. Math. Soc.} {\bf 369} (2017), 1459-1493.
				
				\bibitem{HS24} C. D. Haessig and S. Sperber, Symmetric Power $L$-functions of the hyper-Kloosterman Family,arXiv:2402.13051.
				
				\bibitem{Qin24}Y. Qin, Hodge numbers of motives attached to Kloosterman and Airy moments, {\it J. Reine Angew. Math.} {\bf 808} (2024), 143-192.
				
				\bibitem{Ro86}  P. Robba, Symmetric powers of the $p$-adic Bessel equation, {\it J. Reine Angew. Math.} {\bf 366} (1986), 194-220.
				
				
				\bibitem{WY} C.L. Wang and L.P. Yang, Newton polygons for $L$-functions of generalized kloosterman sums, {\it Forum Math.}
				34 (2022), 77-96.
				
				
				
			\end{thebibliography}
			
		\end{document}